\documentclass[a4paper,12pt,intlimits,oneside]{amsart}

\usepackage{amsmath}
\usepackage{amssymb}
\usepackage{amsfonts}
\usepackage{amsthm,amscd}

\theoremstyle{definition}
\newtheorem{thm}{Theorem}[section]
\newtheorem{prop}[thm]{Proposition}
\newtheorem{cor}[thm]{Corollary}
\newtheorem{lem}[thm]{Lemma}
\newtheorem{defn}[thm]{Definition}
\newtheorem{remark}[thm]{Remark}

\def\E{{\mathfrak E}}

\newcommand{\comment}[1]{}

\numberwithin{equation}{section}

\theoremstyle{definition}

\begin{document}
\title []{ Norm inequalities for Dunkl-type fractional integral and fractional maximal operators in the Dunkl-Fofana spaces}
\author[P. Nagacy]{Pokou Nagacy}
\address{Laboratoire de Math\'ematiques et Applications, UFR Math\'ematiques et Informatique, Universit\'e F\'elix Houphou\"et-Boigny Abidjan-Cocody, 22 B.P 582 Abidjan 22. C\^ote d'Ivoire}
\email{{ pokounagacy@yahoo.com}}
\author[J. Feuto]{Justin Feuto}
\address{Laboratoire de Math\'ematiques et Applications, UFR Math\'ematiques et Informatique, Universit\'e F\'elix Houphou\"et-Boigny Abidjan-Cocody, 22 B.P 1194 Abidjan 22. C\^ote d'Ivoire}
\email{{ justfeuto@yahoo.fr}}
\author[B. A. KPATA]{B\'erenger Akon KPATA}
\address{Laboratoire de Mathématiques et Informatique, UFR des Sciences Fondamentales et
Appliquées,
 Université Nangui Abrogoua,
 02 BP 801 Abidjan 02, Côte d'Ivoire}
\email{{ kpata\_akon@yahoo.fr}}

 \renewcommand{\thefootnote}{}



\renewcommand{\thefootnote}{\arabic{footnote}}
\setcounter{footnote}{0}

\begin{abstract}
We establish some new properties of the Dunkl-Wiener amalgam spaces defined on the real line. These results allow us to obtain the boundedness of Dunkl-type fractional integral and fractional maximal operators in the Dunkl-Fofana spaces.
\end{abstract}

\maketitle

\section{Introduction}
Let us start with some notations that will be used throughout this paper. Let $ k > -\frac{1}{2}$ be a fixed number and $\mu$ be the weighted Lebesgue measure on $\mathbb{R}$, given by
$$d\mu(x)=\left(2^{k+1}\Gamma(k+1)\right)^{-1}\left|x\right|^{2k+1}dx. $$
For $1\leq p\leq\infty$, we denote by $ L^{p}(\mu)$ the Lebesgue space associated with the measure $\mu$, while $L^{p,\, \infty}(\mu)$ is the weak $ L^{p}(\mu)$-space.  We denote by $L^{0}(\mu)$ the complex vector space of equivalence classes (modulo equality $\mu$-almost everywhere) of complex-valued functions $\mu$-measurable  on $\mathbb R$. For $f\in L^{p}(\mu)$, $\Vert f\Vert_p$ stands for the classical norm of $f$.
For any subset $A$ of $\mathbb{R}$, $\chi_{A}$ denotes the characteristic function of $A$. For $x \in \mathbb{R}$ and for $r>0$, we set
$$ B(x,\, r)=\{ y \in \mathbb{R} \, : \, \max \{0,\, |x|-r\}<|y|<|x|+r \}$$  if  $x\neq 0 $ and
$$ B_{r}=B(0,\, r)=(-r, \, r).$$
The letter $C$ will be used for non-negative constants not depending on the relevant variables, and this constant may change from one occurrence to another.

The study of the boundedness properties of certain integral type operators in the topological vector spaces where they act is an important problem in harmonic analysis. Many useful results concerning this topic, which were established in classical Fourier analysis have been generalized within the framework of Dunkl analysis (see, for example \cite{De}, \cite{GM}, \cite{GM3}, \cite{M. G.}, \cite{M-H-S}, and the references therein). Recently, P. Nagacy and J. Feuto \cite{NF} introduced the Dunkl-Wiener amalgam spaces and the Dunkl-Fofana spaces on the real line as described below.
As an application, they established the boundedness of the Hardy-Littlewood maximal operator associated with the Dunkl operator in the the Dunkl-Fofana spaces. Recall that on the real line, the Dunkl operators are differential-difference operators associated with the reflection group $\mathbb{Z}_{2}$ on $\mathbb{R}$.\\
For $1\leq q, \, p\leq \infty$, the Dunkl-Wiener amalgam space $(L^q, \, L^p)(\mu)$ is defined as the subspace of $L^{0}(\mu)$  such that $ \left\|f\right\|_{q,p}=  _{1}\left\|f\right\|_{q,p}<\infty$, where for $r>0$,
 \begin{equation}
 \ _{r}\left\|f\right\|_{q,p}=\left\{\begin{array}{lll}\left\|\left[\int_{\mathbb R}(\tau_{y}\left|f\right|^{q})\chi_{B_r}(x)d\mu(x)\right]^{\frac{1}{q}}\right\|_p&\text{ if }&q<\infty ,\\
 \left\|\left\|f\chi_{B(y,r)}\right\|_{\infty}\right\|_p&\text{ if }&q=\infty,
 \end{array}\right.
 \end{equation}
 with the $L^p(\mu)$-norm taken with respect to the variable $y$. Here and in the sequel, $\tau_{y} (y \in \mathbb{R})$ stands for the Dunkl translation operator (see Section 2 for more details).\\
 The space $(L^q, \, L^p)(\mu)$ is a complex Banach space endowed with the norm $\left\| \cdot \right\|_{q,p} $.\\
Let $1\leq q\leq\alpha\leq p\leq\infty$. The space $\left(L^{q},L^{p}\right)^{\alpha}(\mu)$ is defined as the subspace of $L^{0}(\mu)$  such that $ \left\|f\right\|_{q,p,\alpha}<\infty$, where
\begin{equation*}
\left\|f\right\|_{q,p,\alpha}=\sup_{r>0}(\mu(B_r))^{\frac{1}{\alpha}-\frac{1}{q}-\frac{1}{p}}\ _{r}\left\|f\right\|_{q,p}.
\end{equation*}
It is proved in \cite{NF} that the map $f \mapsto \left\|f\right\|_{q,p,\alpha}$ is a norm on $\left(L^{q},L^{p}\right)^{\alpha}(\mu)$. In addition, For $\alpha \in \{p,\, q \}$, $ \left(L^{q},L^{p}\right)^{\alpha}(\mu)=L^{\alpha} (\mu)$.

In this paper we obtain several new properties of the spaces $(L^q, \, L^p)(\mu)$ by examining their relationship with other topological vector spaces. We also define an equivalent norm to $ \| \cdot \|_{q,\, p}$. Furthermore, we show that $(\left(L^{q},L^{p}\right)^{\alpha}(\mu),\,  \left\|\cdot\right\|_{q,p,\alpha})$ is a complex Banach space. These results allow us to establish norm inequalities in the spaces $\left(L^{q},L^{p}\right)^{\alpha}(\mu)$ for the Dunkl-type fractional integral operator, defined by
\begin{eqnarray*}
I_{\beta} f(x) =  \displaystyle \int_{\mathbb{R}} \tau_{x} f (z) \vert z \vert^{\beta - (2k+2)} d\mu (z),
\end{eqnarray*}
and the fractional maximal operator
$$M_{\beta} f(x)=\sup_{r>0}(\mu(B(0,r))^{\frac{\beta}{2k+2}-1}\int_{B(0,r)}\tau_x\left|f\right|(y)d\mu(y),$$
where $0< \beta < 2k+2 $.\\
In the limiting case $\beta = 0$, the fractional maximal operator reduces to the Hardy-Littlewood maximal operator associated with the Dunkl operator, denoted by $M$.\\
Our main results read as follows.
\begin{thm} \label{theo 01}
Let  $1 < q \leq \alpha \leq p < \infty$ and $0<\beta< \frac{2k+2}{\alpha}$.  Put
$$ \frac{1}{\alpha^{*}} = \frac{1}{\alpha}-\frac{\beta}{2k+2}, \frac{1}{\bar{p}} = \frac{1}{p}(1-\frac{\alpha\beta}{2k+2}) \text{ and } \frac{1}{\bar{q}} = \frac{1}{q}(1-\frac{\alpha\beta}{2k+2}).$$
 Then
\begin{equation*}
\Vert I_{\beta}f \Vert_{\bar{q},\bar{p},\alpha^{*}} \leq C \Vert f \Vert^{1-\frac{\alpha\beta}{2k+2}}_{q,p,\alpha} \Vert f \Vert^{\frac{\alpha\beta}{2k+2}}_{q,\infty,\alpha},\;\;\;\;\;\;f \in (L^{q}, L^{p})^{\alpha} (\mu)
\end{equation*}
 and
 \begin{equation*}
\Vert I_{\beta}f \Vert_{\bar{q},\bar{p},\alpha^{*}} \leq C \Vert f \Vert_{q,p,\alpha},\;\;\;\;\;\; f \in (L^{q}, L^{p})^{\alpha} (\mu).
\end{equation*}
\end{thm}
For $q=1$, we prove the following
\begin{thm} \label{theo 2}
Let $ 1 < \alpha < p < \infty$, $0<\beta< \frac{2k+2}{\alpha}$. Put $$ \frac{1}{\alpha^{*}} = \frac{1}{\alpha}-\frac{\beta}{2k+2}, \frac{1}{\bar{p}} = \frac{1}{p}(1-\frac{\alpha\beta}{2k+2}) \text{ and } \frac{1}{\bar{q}} = 1-\frac{\alpha\beta}{2k+2}. $$
Then
\begin{equation*}
\Vert I_{\beta}f \Vert_{(L^{\bar{q},\infty}, L^{\bar{p}})^{\alpha^{*}} (\mu) } \leq C_{3} \Vert f \Vert^{1-\frac{\alpha\beta}{2k+2}}_{1,p,\alpha} \Vert f \Vert^{\frac{\alpha\beta}{2k+2}}_{1,\infty,\alpha}, \;\;\;\;\;\;\; f \in (L^{1}, L^{p})^{\alpha} (\mu),
\end{equation*}
where
$$
 \Vert I_{\beta}f \Vert_{(L^{\bar{q},\infty},\; L^{\bar{p}})^{\alpha^{*}} (\mu) } = \sup\limits_{r > 0} \mu (B(0,r))^{\frac{1}{\alpha^{*}} - \frac{1}{\bar{q}} - \frac{1}{\bar{p}}} \Vert \Vert I_{\beta}f \tau^{\frac{1}{\bar{q}}}_{(- \cdot)}\chi_{B_r} \Vert_{L^{\bar{q},\infty}(\mu)}   \Vert_{\bar{p}}
$$
and
\begin{equation*}
\Vert I_{\beta}f \Vert_{(L^{\bar{q},\infty}, \, L^{\bar{p}})^{\alpha^{*}} (\mu) } \leq C \Vert f \Vert_{1,p,\alpha},\;\;\;\;\;\;\;\; f \in (L^{1}, L^{p})^{\alpha} (\mu).
\end{equation*}
\end{thm}
As a consequence of Theorems \ref{theo 01}, we obtain the corresponding boundedness results for $M_{\beta}$.
\begin{cor}\label{cor 1}
Let  $1 < q \leq \alpha \leq p < \infty$, $0<\beta< \frac{2k+2}{\alpha}$. Put
$$\frac{1}{\alpha^{*}} = \frac{1}{\alpha}-\frac{\beta}{2k+2},\;\;\; \frac{1}{\bar{p}} = \frac{1}{p}(1-\frac{\alpha\beta}{2k+2}) \text{ and } \frac{1}{\bar{q}} = \frac{1}{q}(1-\frac{\alpha\beta}{2k+2}).$$
Then
$$\Vert M_{\beta}f \Vert_{\bar{q},\bar{p},\alpha^{*}} \leq C \Vert f \Vert_{q,p,\alpha}, \;\;\;\;\; f \in (L^{q}, L^{p})^{\alpha} (\mu).$$
\end{cor}
As a consequence of Theorem \ref{theo 2}, we obtain the following weak-type inequality for $M_{\beta}$.
\begin{cor}
\label{cor 2}
Let  $1 < \alpha < p < \infty$, $0<\beta< \frac{2k+2}{\alpha}$. Put
$$\frac{1}{\alpha^{*}} = \frac{1}{\alpha}-\frac{\beta}{2k+2},\;\;\; \frac{1}{\bar{p}} = \frac{1}{p}(1-\frac{\alpha\beta}{2k+2}) \text{ and } \frac{1}{\bar{q}} = \frac{1}{q}(1-\frac{\alpha\beta}{2k+2}).$$
Then
 $$\Vert M_{\beta}f \Vert_{(L^{\bar{q}, \infty}, L^{\bar{p}})^{\alpha^{*}} (\mu) } \leq C \Vert f \Vert_{1,p,\alpha} , \;\;\;\;\;\; f \in (L^{1}, L^{p})^{\alpha} (\mu).$$
\end{cor}
\begin{remark} Note that
\begin{enumerate}
\item Theorem \ref{theo 01} and Corollary \ref{cor 1} in the case $\alpha=p$ or $\alpha=q$  were proved in \cite{GM}.
\item Theorem \ref{theo 01} and Theorem \ref{theo 2} are generalizations, within the framework of Dunkl analysis, of Theorem 4.5 and Corollary 4.4 established in \cite{Fe} in the one-dimensional case.
\end{enumerate}
\end{remark}
The remainder of this note is organized as follows. Section 2 contains  the prerequisite for Dunkl analysis on the real line. In Section 3 we gather some basic results about the spaces $(L^{q},\, L^{p})(\mu)$. New properties of the spaces $(L^{q},\, L^{p})(\mu)$ and $(L^{q},\, L^{p})^{\alpha}(\mu)$ are established in Section 4. Section 5 is devoted to the proof of Theorems \ref{theo 01} and \ref{theo 2} as well as Corollaries \ref{cor 1} and \ref{cor 2}.

\section{Prerequisite on Dunkl analysis on the real line}
The Dunkl operator associated with the reflection group $\mathbb Z_{2}$ on $\mathbb R$ is defined by
$$\Lambda_{k}f(x)=\frac{df}{dx}(x)+\frac{2k+1}{x}\left(\frac{f(x)-f(-x)}{2}\right).$$
For $\lambda\in\mathbb C$, the Dunkl kernel denoted by $\E_{k}(\lambda)$ (see \cite{D2}), is the only solution of the initial value problem
$$\Lambda_{k}f(x)=\lambda f(x),\ f(0)=1,\ x\in\mathbb R.$$
It is given by the formula
\begin{equation*}
\E_{k}(\lambda x)=j_{k}(i\lambda x)+\frac{\lambda x}{2(k+1)}j_{k+1}(i\lambda x),\ x\in\mathbb R,
\end{equation*}
where 
$$j_k(z)=2^k\Gamma(k+1)\frac{J_k(z)}{z^k}=\Gamma(k+1)\sum^{\infty}_{n=0}\frac{(-1)^nz^{2n}}{n!2^{2n}\Gamma(n+k+1)},\quad z\in \mathbb C$$
is the normalized Bessel function of the first kind and of order $k$.
Notice that $\Lambda_{-\frac{1}{2}}=\frac{d}{dx}$ and $\E_{-\frac{1}{2}}(\lambda x)=e^{\lambda x}$.
 It is also proved (see \cite{R4}) that $\left|\E_{k}(ix)\right|\leq 1$ for every $x\in\mathbb R$.

The Dunkl kernel $\E_{k}$ gives rise to an integral transform on $\mathbb R$ denoted $\mathcal F_{k}$ and called Dunkl transform (see \cite{MFE}). For $f\in L^{1}(\mu)$
 $$\mathcal{F}_{k}f(\lambda)=\int_{\mathbb R}\mathfrak E_{k}(-i\lambda x)f(x)d\mu(x),\ \ \lambda\in\mathbb R.$$
We have the following properties of the Dunkl transform  given in  \cite{MFE}  (see also \cite{D1}).
\begin{prop} \label{theo 1}
\begin{enumerate}
\item Let $f\in L^1(\mu)$. If $\mathcal{F}_{k} (f)$ is in $L^1(\mu)$, then we have the following inversion formula :
$$f(x)=c\int_{\mathbb R}\E_{k}(ixy)\mathcal{F}_{k}(f)(y)d\mu(y).$$
\item The Dunkl transform has a unique extension to an isometric isomorphism on $L^2(\mu)$.
\end{enumerate}
\end{prop}

\begin{defn}
Let $y\in\mathbb R$ be given. The generalized translation operator $f\mapsto\tau_y f$ is defined on $L^2(\mu)$ by the equation
$$\mathcal{F}_{k}(\tau_yf)(x)=\E_{k}(ixy)\mathcal{F}_{k}(f)(x);\quad x\in\mathbb R$$
\end{defn}
Mourou proved in \cite{M} that generalized translation operators have the following properties:
\begin{prop}
\begin{enumerate}
\item The operator $\tau_{x}$, $x\in\mathbb R$, is a continuous linear operator from $\mathcal E(\mathbb R)$ into itself,
\item For $f\in\mathcal E(\mathbb R)$ and  $x,y\in \mathbb R$
\begin{enumerate}
\item $\tau_{x}f(y)=\tau_{y}f(x)$,
\item $\tau_{0}f=f$,
\item $\tau_{x}\circ\tau_{y}=\tau_{y}\circ\tau_{x}$.
\end{enumerate}
\end{enumerate}
\end{prop}

\begin{prop} [see Soltani \cite{S}]\label{prop 1}
	For  $x,\lambda \in \mathbb{R}$ and $f \in  L^{1}(\mu)$,
	$$\mathcal{F}_{k}(\tau_{x}f)(\lambda) = \E_{k}(ix \lambda) \mathcal{F}_{k}(f)(\lambda).$$
\end{prop}

For $x\in\mathbb R$, let $B(x,r)=\left\{y\in\mathbb R:\max\left\{0,\left|x\right|-r\right\}<\left|y\right|<\left|x\right|+r\right\}$ if $x\neq 0$ and $B_r:=B(0,r)=\left]-r,r\right[$. We  have $\mu(B_r)=d_{k} r^{2k+2}$ where $ d_{k}=\left[2^{k+1}(k+1)\Gamma(k+1)\right]^{-1}$.

 For $x\in\mathbb R $ and $r>0$, the map $y\mapsto\tau_{x}\chi_{B_r}(y)$  is supported in $B(x,r)$ and
 \begin{equation}
 0\leq \tau_{x}\chi_{B_r}(y)\leq\min\left\{1,\frac{2C_\kappa}{2\kappa+1}\left(\frac{r}{\left|x\right|}\right)^{2\kappa+1}\right\},\  y \in B(x,r),\label{**}
 \end{equation}
 as proved in \cite{GM3}.\\
 The following property of the translation will be useful.
 \begin{prop}
There exists positive real numbers $C_{1}$ and $C_{2}$ such that for all $(y, \, f) \in \mathbb{R} \times L^{0}(\mu)$ we have
\begin{eqnarray*}
C_{1} \displaystyle \int_\mathbb{R}   \vert f \vert^{q} (x) \chi_{I(y,1)} (x)  d\mu (x)&\leq & \displaystyle \int_\mathbb{R}  \tau_{y} \vert f \vert^{q} (x) \chi_{B(0,1)} (x)  d\mu (x)\\
& \leq & C_{2} \displaystyle \int_\mathbb{R}   \vert f \vert^{q} (x) \chi_{I(y,1)} (x)  d\mu (x).
 \end{eqnarray*}
\label{p2}
 \end{prop}

Let $f$ and $g$ be two continuous functions on $\mathbb R$ with compact support. We
define the generalized convolution $\ast_{k}$ of $f$ and $g$ by
\begin{equation*}
f\ast_{k}g(x)=\int_{\mathbb R}\tau_{x}f(-y)g(y)d\mu(y).
\end{equation*}
The generalized convolution $\ast_{k}$ is associative and commutative (see \cite{R4}). We also have the below result.
\begin{prop}[see Soltani \cite{S}]\label{young}
\begin{enumerate}
\item For all $x\in\mathbb R$, the operator $\tau_{x}$ extends to $L^{p}(\mu)$, $p\geq 1$, and
\begin{equation}
\left\|\tau_{x}f\right\|_{p}\leq 4\left\|f\right\|_{p}\label{conttranslation}
\end{equation}
for all $f\in L^{p}(\mu)$.
\item Assume that $p, q, r \in\left[1,\infty\right]$ and satisfy $\frac{1}{p} +\frac{1}{q} = 1+\frac{1}{r}$. Then the generalized convolution defined on
$\mathcal C_{c}\times\mathcal C_{c}$, extends to a continuous map from $L^{p}(\mu)\times L^{q}(\mu)$ to $L^{r}(\mu)$, and we have
\begin{equation*}
\left\|f\ast_{k}g\right\|_{r}\leq 4\left\|f\right\|_{p}\left\|g\right\|_{q}.
\end{equation*}
\end{enumerate}
\end{prop}
It is also proved in \cite{GM2} that if $f\in L^1(\mu)$ and $g\in L^p(\mu)$, $1\leq p<\infty$, then
\begin{equation*}
\tau_x(f\ast_{k} g)=\tau_x f\ast_{k} g=f\ast_{k} \tau_x g,x\in\mathbb{R}.
\end{equation*}
For any measurable function $f$ on $\mathbb{R}$, we define its
distribution and rearrangement functions

$$ d_{f,k}(\lambda) = \mu \left(  \left\{
\begin{array}{ll}
x \in \mathbb{R} : \vert f(x) \vert > \lambda
\end{array}
\right\} \right), $$
$$ f_{k}^{*} (\lambda) = \inf \left\{
\begin{array}{ll}
t > 0 : d_{f,k}(t) \leq \lambda
\end{array}
\right\}.  $$
For $1 \leqslant p \leqslant \infty$ and $1 \leqslant q \leqslant \infty$, define:

\begin{equation*}
\left\|f\right\|_{L^{p,\, q}(\mu)} =\left\{\begin{array}{lll}\left[\int_{0}^{\infty}(s^{\frac{1}{p}} f_{k}^{*} (s))^{q} \frac{ds}{s} \right]^{\frac{1}{q}} &\text{ si } & 1\leq p, \, q<\infty\\
\sup\limits_{s>0} s^{\frac{1}{p}} f_{k}^{*} (s) & \text{ si }&  1\leq p\leq \infty, q=\infty.
\end{array}\right.
\end{equation*}
The generalized Lorentz spaces $L^{p, \, q}(\mu)$ is defined as the set
of all measurable functions $f$ such that $\left\|f\right\|_{L^{p, \, q}(\mu)} < \infty$.(See \cite{M. G.}, \cite{Hat} ).\\
A particular case of this space is the weak Lebesgue space $L^{p,\, \infty}(\mu)$.

\begin{prop} [see \cite{M. G.} ] \label{prop 37}
	If $f \in L^{p_{1},\, q_{1}}(\mu)$, $g \in L^{p_{2}, \, q_{2}}(\mu)$ and $\frac{1}{p_{1}} + \frac{1}{p_{2}} > 1$, then $f*_{k}g \in L^{p_{3},\, q_{3}}(\mu)$ where $\frac{1}{p_{3}} +1 = \frac{1}{p_{1}} + \frac{1}{p_{2}}$ and $q_{3} \ge 1$ is any number such that  $ \frac{1}{q_{3}} \leq \frac{1}{q_{1}} + \frac{1}{q_{2}}$. \\ Moreover, $$ \Vert f*_{k}g \Vert_{L^{p_{3},\, q_{3}}(\mu)} \leq 3p_{3}\Vert f \Vert_{L^{p_{1},\, q_{1}}(\mu)} \Vert g \Vert_{L^{p_{2},\, q_{2}}(\mu)}. $$
\end{prop}

\section{Preliminaries results on the Dunkl-Wiener amalgam spaces}
The propositions of this section are proved in \cite{NF}. \\
H\"older inequality remains valid in the Dunkl-Wiener amalgam spaces as stated below.
\begin{prop}  \label{prop 4}
	Let $1\leq q_{1} ,\,  p_{1}\leq\infty$ and $1\leq q_{2} ,\,  p_{2}\leq \infty$ such that
	$$\frac{1}{q_{1}} + \frac{1}{q_{2}} = \frac{1}{q} \leq 1\text{ and }\frac{1}{p_{1}} + \frac{1}{p_{2}} = \frac{1}{p} \leq 1.$$
	If $f\in(L^{q_{1}}, \, L^{p_{1}}) (\mu)$ and $g\in(L^{q_{2}}, \, L^{p_{2}})(\mu) $ then $fg\in(L^{q}, \, L^{p}) (\mu)$.\\
 Moreover,
	$$\left\| fg \right\|_{q, \, p} \leq \left\| f \right\|_{q_{1}, \, p_{1}}  \left\| g \right\|_{q_{2}, \, p_{2}}.$$
\end{prop}
The following result shows that $L^{s}(\mu)\subset (L^{q}, \, L^{p})(\mu)$ when $1\leq q\leq s\leq p\leq\infty$.
\begin{prop}  \label{inclaqp}Let $1\leq q\leq s\leq p\leq\infty$. We have
\begin{equation}
\left\|f\right\|_{q, \, p}\leq  4^{\frac{1}{q}}\left\|f\right\|_{s}\mu(B_1)^{\frac{1}{p}-\frac{1}{s}+\frac{1}{q}}.
\end{equation}
\end{prop}
The family of spaces
$(L^{q}, \, L^{p})(\mu)$ is decreasing with respect to the $q$-power. More precisely, we have the following result.
\begin{prop}\label{include1}
Let $1\leq q_1\leq q_2 \leq\infty$ and $1\leq p\leq\infty$. For any locally $\mu$-integrable function $f$, we have
\begin{equation}\left\|f\right\|_{q_1,p}\leq\mu(B_1)^{\frac{1}{q_1}-\frac{1}{q_2}}\left\|f\right\|_{q_2,p}.
\end{equation}
\end{prop}

\section{New properties of the spaces $ (L^{p},\, L^{p})(\mu)$ and $(L^{q},\, L^{p})^{\alpha}(\mu)$}
For $(x,\, r) \in \mathbb{R} \times (0,\, \infty)$ and for  $\delta >0$,  we set
$$ I(x,\, r)=(x-r, \, x+r)\;\;\text{ and } \;\; \delta I=I(x,\, \delta r).$$
We will examine the relationship between the Dunkl-Wiener amalgam spaces and other topological vector spaces.
\begin{prop} \label{prop 3}
Let $1\leq p\leq \infty$. Then
$$\left\|f\right\|_{p,\, p} = \mu(B_1)^{\frac{1}{p}} \left\|f\right\|_{p}$$
and consequently, $(L^{p},\, L^{p})(\mu)=L^{p}(\mu)$.
\end{prop}

\begin{proof}
The case where $p= \infty$ is proved in \cite{NF}. So we suppose that $p< \infty$. We have
\begin{eqnarray*}
	\Vert f \Vert_{p,\, p}
	&=& \left[ \displaystyle \int_\mathbb{R} \displaystyle \int_\mathbb{R}   \vert f \vert^{p} (x) \tau_{-y} \chi_{B(0,1)} (x)  d\mu (x) d\mu(y) \right]^{\frac{1}{p}} \\
	&=& \left[ \displaystyle \int_\mathbb{R} \left[ \displaystyle \int_\mathbb{R} \tau_{x} \chi_{B(0,1)} (y) d\mu(y) \right]  \vert f \vert^{p} (x)   d\mu (x) \right]^{\frac{1}{p}} \\
	&=& (\mu(B_1))^{\frac{1}{p}} \left[ \displaystyle \int_\mathbb{R}  \vert f \vert^{p} (x)   d\mu (x) \right]^{\frac{1}{p}}.
\end{eqnarray*}
\end{proof}
Proposition \ref{include1}  and Proposition \ref{prop 3} yield the following remark.
\begin{remark}
For $1\leq p \leq q \leq \infty$, we have $ (L^{q},\, L^{p})(\mu) \subset L^{p}(\mu)$.
\label{r1}
\end{remark}
Let us denote by $\mathcal{S}(\mathbb{R})$ the Schwartz space. We endow it with the topology induced by the family of semi-norms given by
$$ \mathcal{N}_{m}(f) =\sup\limits_{x \in \mathbb{R}} (1 + \vert x \vert)^{m} \sum_{n\leq m}  \left\vert \dfrac{d^{(n)}}{dx^{n}} f(x) \right\vert, \, \,\,\,\,\,\,\,\:f \in \mathcal{S}(\mathbb{R}), $$
where $m$ and $n$ are non-negative integers.

\begin{prop}
Let $1 \leq q, \, p \leq \infty$.  Then	$\mathcal{S}(\mathbb{R})$ is continuously embedded in $(L^{q}, L^{p})(\mu)$.
\end{prop}

\begin{proof}
	Let $m$ be a non-negative integer and $f \in \mathcal{S}(\mathbb{R})$. We have
	
	\begin{eqnarray}\label{rel 15}
	\vert f(x) \vert \leq (1 + \vert x \vert)^{-m} \mathcal{N}_{m}(f), \quad x \in \mathbb{R}
	\end{eqnarray}
$\bullet$ Suppose that $q<\infty$.\\
 Let $y \in \mathbb{R}$. It follows from (\ref{rel 15}) that
		 $$ \vert f(x) \vert \tau^{\frac{1}{q}}_{-y}\chi_{B_1}(x) \leq (1 + \vert x \vert)^{-m} \mathcal{N}_{m}(f)\tau^{\frac{1}{q}}_{-y}\chi_{B_1}(x), \;\;\;\;\;\;\;x \in \mathbb{R}. $$
Therefore,
		\begin{eqnarray} \label{rel 0}
		\Vert f  \tau^{\frac{1}{q}}_{-y}\chi_{B_1} \Vert_{q}  \leq \mathcal{N}_{m}(f) \Vert (1 + \vert  \cdot \vert)^{-m} \tau^{\frac{1}{q}}_{-y}\chi_{B_1} \Vert_{q}
		\end{eqnarray}
Since
		\begin{equation}\label{rel 14}
		\frac{1}{4}(1+ \vert x \vert) \leq 1+ \vert y \vert \leq 4(1+ \vert x \vert),\;\;\;\;\;\;\;\;\; y \in \mathbb{R},\;x \in B(y,1),
		\end{equation}
		We get from (\ref{rel 0}) that
			\begin{eqnarray*}
				\Vert f  \tau^{\frac{1}{q}}_{-y}\chi_{B_1} \Vert_{q}  &\leq& 4^{m}\mathcal{N}_{m}(f) (1 + \vert  y \vert)^{-m} \Vert  \tau^{\frac{1}{q}}_{-y}\chi_{B_1} \Vert_{q} \\
				&\leq& 4^{m}\mathcal{N}_{m}(f) (1 + \vert  y \vert)^{-m} [\mu(B_{1})]^{\frac{1}{q}}.
			\end{eqnarray*}
			according to \cite[Theorem 6.3.4]{DX}.
			Hence, for all $m> \frac{2k+2}{p}$,
			$$\Vert f \Vert_{q,p} = \Vert \Vert f  \tau^{\frac{1}{q}}_{(- \cdot)}\chi_{B_1} \Vert_{q} \Vert_{p} \leq
			4^{m}\mathcal{N}_{m}(f) [\mu(B_{1})]^{\frac{1}{q}} \Vert  (1 + \vert \cdot \vert)^{-m}\Vert_{p} < \infty.$$
Therefore, for all $m> \frac{2k+2}{p}$,
$$\Vert f \Vert_{q,p} \leq C(q,m,k,p) \mathcal{N}_{m}(f).$$
$\bullet$ Suppose that $q= \infty$. \\
Let $y \in \mathbb{R}$. From (\ref{rel 15}) we have	
		$$\vert f(x) \chi_{B(y,1)}(x)\vert \leq (1 + \vert x \vert)^{-m} \mathcal{N}_{m}(f)\chi_{B(y,1)}(x),\,\,\,\,\,\, x \in \mathbb{R}.$$
It follows from (\ref{rel 14}) that
		\begin{eqnarray*}
			\Vert f\chi_{B(y,1)}\Vert_{\infty} &\leq& \mathcal{N}_{m}(f) \Vert (1 + \vert \cdot \vert)^{-m} \chi_{B(y,1)}\Vert_{\infty}\\
			&\leq&  4^{m}\mathcal{N}_{m}(f)(1 + \vert y \vert)^{-m}  \Vert \chi_{B(y,1)}\Vert_{\infty}\\
& = & 4^{m}\mathcal{N}_{m}(f)(1 + \vert y \vert)^{-m}.
		\end{eqnarray*}
Hence, for all  $m> \frac{2k+2}{p}$,
 $$\Vert f \Vert_{\infty,p} = \Vert \Vert f\chi_{B(y,1)}\Vert_{\infty} \Vert_{p} \leq 4^{m}\mathcal{N}_{m}(f)\Vert (1 + \vert \cdot \vert)^{-m} \Vert_{p} < \infty. $$
 Therefore, for all $m> \frac{2k+2}{p}$, $$\Vert f \Vert_{q, \, p} \leq C(q, \, m, \, k, \, p) \mathcal{N}_{m}(f)$$ and
the result follows.	
\end{proof}
The following result asserts that for $1 \leq q <s \leq p< \infty$ we have $L^{s, \; \infty}(\mu) \subset (L^{q}, \,  L^{p}) (\mu)$. \begin{prop}\label{prop 38}
Let $1 \leq q <s \leq p<\infty$. Then there exists a constant $C=C(p,\, q, \, s)$ such that
$$ \Vert f \Vert_{q,\, p} \leq C \Vert f \Vert_{L^{s, \, \infty}(\mu)}, \quad \,\,\,\,\,\, \: f \in L^{0}(\mu).$$
\end{prop}
\begin{proof}
	Let $f \in L^{0}(\mu)$ and $1 \leq q <s \leq p<\infty$. \\
The result is obvious if $f$ is not an element of $L^{s, \infty}(\mu)$. So we suppose that $f \in L^{s, \infty}(\mu)$. \\ We have
	$$ \frac{ps}{qs-qp+ps} > 1 \quad \text{and} \quad \frac{qs-pq+ps}{ps} + \frac{q}{s} = \frac{q}{p} +1 $$
so that	according to Proposition \ref{prop 37},
	\begin{eqnarray*}
		\Vert f \Vert^{q}_{q,p} &=& \Vert \vert f \vert^{q} *_{k} \chi_{B(0,1)} \Vert_{L^{\frac{p}{q}}(\mu)}\\
		&\leq & \frac{3p}{q} \mu(B(0,1))^{\frac{qs-qp+ps}{ps}} (\Vert f \Vert_{L^{s, \infty}(\mu)} )^{q}.
	\end{eqnarray*}	
\end{proof}

\begin{prop} \label{prop 50}
	Assume that
	$1 \leq r_{1}, \, r_{2}, \; s_{1}, \, s_{2} \leq \infty $. Let $\theta \in (0,\,  1)$ satisfying
\begin{equation}
\frac{1}{r} = \dfrac{ \theta}{r_{1}} + \frac{1-\theta}{r_{2}}\;\; \text{ and } \;\;  \frac{1}{s} = \frac{ \theta}{s_{1}} + \dfrac{1-\theta}{s_{2}}.
\label{eq1}
\end{equation}
	Then, for all $f \in L^{0} (\mu)$ we have
 $$\Vert f \Vert_{s,r} \leq \Vert f \Vert^{\theta}_{s_{1},r_{1}} \Vert f \Vert^{1- \theta}_{s_{2},r_{2}} .$$ 				
\end{prop}
\begin{proof}
Let $1 \leq r_{1}, \, r_{2}, \; s_{1}, \, s_{2} \leq \infty$ and $\theta \in (0,\, 1)$ satisfying (\ref{eq1}). Let $f \in L^{0}(\mu)$.\\
$\bullet$ Suppose that $s_{1}< \infty$ and $s_{2} < \infty $ .\\
		Let $y \in \mathbb{R}$. By H\"older inequality we have
\begin{eqnarray*}
& & \displaystyle \int_{\mathbb R}\left|f\right|^{s} (x) \tau_{-y}\chi_{B_1}(x)d\mu(x) \\
		&=& \displaystyle \int_{\mathbb R}\left|f\right|^{s(1- \theta)+s\theta} (x) (\tau_{-y}\chi_{B_1})^{\frac{ s\theta}{s_{1}} + \frac{s(1-\theta)}{s_{2}}}(x)d\mu(x) \\
  & \leq & \left(\displaystyle \int_{\mathbb R} \left( \left|f\right|^{s\theta} (x) (\tau_{-y}\chi_{B_1})^{\frac{ s\theta}{s_{1}}} (x)  \right)^{\frac{s_{1} }{ s\theta}} d\mu(x) \right)^{\frac{s\theta }{s_{1} }} \times \\
 & &  \left(\displaystyle \int_{\mathbb R} \left( \left|f\right|^{s(1- \theta)} (x) (\tau_{-y}\chi_{B_1})^{\frac{s(1-\theta)}{s_{2}}} (x)  \right)^{\frac{s_{2} }{ s(1-\theta)}} d\mu(x) \right)^{\frac{s(1-\theta) }{ s_{2}}}.
\end{eqnarray*}
By applying again H\"older inequality,  it follows that
\begin{eqnarray*}
& & \left\Vert \left( \displaystyle \int_{\mathbb R}\left|f\right|^{s} (x) \tau_{-(\cdot)}\chi_{B_1}(x)d\mu(x) \right)^{\frac{1}{s}} \right\Vert_{r}\\
& \leq & \left\Vert \left(\displaystyle \int_{\mathbb R}  \left|f\right|^{s_{1}} (x) (\tau_{-(\cdot)}\chi_{B_1}) (x) d\mu(x) \right)^{\frac{\theta }{s_{1} }} \right\Vert^{\theta}_{r_{1}} \times  \\
& &  \left\Vert \left(\displaystyle \int_{\mathbb R} \left|f\right|^{s_{2}} (x) (\tau_{-(\cdot)}\chi_{B_1}) (x) d\mu(x) \right)^{\frac{1-\theta }{ s_{2}}} \right\Vert^{1-\theta}_{r_{2}}.
\end{eqnarray*}
 Hence,
		$$\Vert f \Vert_{s,\, r} \leq \Vert f \Vert^{\theta}_{s_{1},\, r_{1}} \Vert f \Vert^{1- \theta}_{s_{2}, \, r_{2}}.$$
$\bullet$ Suppose that $s_{1} = s_{2} = s = \infty$ .\\
If $r=\infty$ then $r_{1}=r_{2}=\infty$. In this case, the result is obvious. So we suppose that $r< \infty$.\\
It follows from H\"older inequality that
		\begin{eqnarray*}
			\Vert f \Vert_{\infty, r}^{r}
			&=& \displaystyle \int_{\mathbb R} \Vert f \chi_{B(y,1)} \Vert^{ r\theta}_{\infty} \Vert f \chi_{B(y,1)} \Vert^{r(1-\theta)}_{\infty} d\mu(y) \\
			&\leq & \Vert \Vert f \chi_{B(\cdot,1)} \Vert_{\infty} \Vert_{r_{1}}^{r\theta} \Vert \Vert f \chi_{B(\cdot,1)} \Vert_{\infty} \Vert_{r_{2}}^{r(1-\theta)}.
		\end{eqnarray*}
Hence,  $$\Vert f \Vert_{\infty,r} \leq \Vert f \Vert^{\theta}_{\infty,r_{1}} \Vert f \Vert^{1- \theta}_{\infty,r_{2}}.$$
$\bullet$ Suppose that  $s_{1} = \infty$ and  $s_{2} < \infty$.\\
From H\"older inequality, we have
		\begin{eqnarray*}
			\Vert f \Vert_{s, r} &=& \left\Vert \left( \displaystyle \int_{\mathbb R}\left|f\right|^{s(1-\theta)+ s\theta} (x) \chi_{B(\cdot,1)}(x) \tau_{-(\cdot)}\chi_{B_1}(x)d\mu(x) \right)^{\frac{1}{s}} \right\Vert_{r} \\
			&\leq& \left\Vert  \Vert f \chi_{B(\cdot,1)} \Vert^{\theta}_{\infty}    \left( \displaystyle \int_{\mathbb R}\left|f\right|^{s_{2}} (x)  \tau_{-(\cdot)}\chi_{B_1}(x)d\mu(x) \right)^{\frac{1-\theta}{s_{2}}}  \right\Vert_{r} \\
			&\leq& \Vert  \Vert f \chi_{B(\cdot,1)} \Vert_{\infty}  \Vert^{\theta}_{r_{1}} \left\Vert \left( \displaystyle \int_{\mathbb R}\left|f\right|^{s_{2}} (x)  \tau_{-(\cdot)}\chi_{B_1}(x)d\mu(x) \right)^{\frac{1}{s_{2}}}  \right\Vert^{1-\theta}_{r_{2}}.
		\end{eqnarray*}
Hence,
$$\Vert f \Vert_{s,r} \leq \Vert f \Vert^{\theta}_{\infty,r_{1}} \Vert f \Vert^{1- \theta}_{s_{2},r_{2}}.$$
$\bullet$ Suppose that $s_{1} < \infty$ and $s_{2} = \infty$.\\
This case is similar to the previous one.	
\end{proof}

Now, we focus on the existence of an equivalent norm to $\left\| \cdot \right\|_{q,\, p}$. For this purpose, we set
$$Q_{l} = [l,\, l+1), \;\;\;\;\;\;\;\; l \in \mathbb{Z}.$$
 It is clear that $\{Q_{l}\, : \, l \in \mathbb{Z}  \}$ is a disjoint family which covers $\mathbb{R}$.\\
 For any $\delta>0$ and for any integer $l$, we denote by $\delta Q_{l}$ the half-open interval $[a_{l}, \, b_{l})$, where $a_{l}=l+\frac{1-\delta}{2}$ and $ b_{l}=l+\frac{1+\delta}{2}$.
\begin{lem} \label{lem 20}
Let $l$ be an integer. The following assertions hold.
	\begin{enumerate}
		\item For all $y \in \mathbb{R}$, the set $L_{y}=\left\{
		\begin{array}{ll}
		i  \in \mathbb{Z} : Q_{i}\cap I(y,1) \ne \emptyset
		\end{array}
		\right\}$ has at most three elements.
		\item For all $y \in Q_{l}$ , $Q_{l} \subset I(y,1)\subset 3Q_{l}$.
		\item $\mu$ is doubling with respect to $Q_{l}$.
	\end{enumerate}	
\end{lem}
\begin{proof}
We give the proof of the last assertion because the others are trivial.\\
Let $l$ be an integer. Put $u=l +\frac{1}{2}$. It is proved in \cite{NF} that for any $(x,\, r)\in \mathbb{R}\times (0,\, \infty)$, $\mu$ is doubling for $I(x,\, r)$. Then there exists a constant $C>0$ such that $$\mu(2I(u,\frac{1}{2}))   \leq C \mu(I(u,\, \frac{1}{2})).$$
As $$\mu(2Q_{l})  = \mu(2I(u,\, \frac{1}{2}))\;\;\;\text{and} \;\;\;\; \mu(I(u,\, \frac{1}{2})) =  \mu(Q_{l}),$$ we deduce that $$\mu (2Q_{l}) \leq C  \mu(Q_{l}).$$
This ends the proof.
\end{proof}

For $1\leq p, \, q \leq \infty$ and for any $\mu$-measurable function $f$ on $\mathbb{R}$, set
\begin{equation*} \label{rel 2}
		\widetilde{\left\|f\right\|}_{q,p}=\left\{\begin{array}{lll}\left(\sum\limits_{l \in \mathbb{Z} } \mu (Q_{l} ) \left\Vert f \chi_{Q_{l}} \right\Vert^{p}_{q}\right)^{\frac{1}{p}}&\text{ if }&p<\infty , \\
			\sup\limits_{l \in \mathbb{Z}} \left\Vert f \chi_{Q_{l}} \right\Vert_{q}&\text{ si }&p=\infty.
		\end{array}\right.
		\end{equation*}
\begin{prop} \label{prop 45}
Let $1\leq p, \, q\leq  \, \infty$. Then, there are positive constants $C_{1}$ and $C_{2}$ such that
	$$C_{1}\widetilde{\left\|f\right\|}_{q,\, p} \leq  \left\|f\right\|_{q, \, p} \leq C_{2}  \widetilde{\left\|f\right\|}_{q,\, p},\,\,\,\,\,\,f \in L^{0}(\mu).$$
\end{prop}
\begin{proof}
Let $f$ be any $\mu$-measurable function on $\mathbb{R}$.\\
$1^{rst}$ case.  We suppose that $q< \infty$.\\
According to Proposition \ref{p2}, there exists a constant $C > 0$ not depending on $f$ such that
$$\displaystyle \int_\mathbb{R}  \tau_{y} \vert f \vert^{q} (x) \chi_{B(0,1)} (x)  d\mu (x) \leq C \displaystyle \int_\mathbb{R}   \vert f \vert^{q} (x) \chi_{I(y,1)} (x)  d\mu (x).$$
$\bullet$ Suppose $p< \infty$.
Then
\begin{eqnarray*}
	\left\|f\right\|_{q, \, p}
	& \leq& C(q) \left[
	\displaystyle \int_\mathbb{R} \left[ \displaystyle \int_\mathbb{R}   \vert f \vert^{q} (x) \chi_{I(y,1)} (x)  d\mu (x) \right]^{\frac{p}{q}} d\mu(y) \right]^{\frac{1}{p}} \text{}      \\
	& \leq & C(q)\left[ \sum_{l \in \mathbb{Z}} \displaystyle \int_{Q_{l}} \left[\sum_{i \in L_{y}}  \int_{Q_{i}} \vert f \vert^{q} (x) \chi_{I(x,1)} (y)  d\mu (x) \right]^{\frac{p}{q}} d\mu(y) \right]^{\frac{1}{p}}.
\end{eqnarray*}	
By Lemma \ref{lem 20}, we get
\begin{eqnarray*}
	\left\|f\right\|_{q,p} & \leq & C(q,\, p) \left[  \sum\limits_{i \in \mathbb{Z} }\sum\limits_{l \in \mathbb{Z}}
	\displaystyle \int_{Q_{l}\cap 3Q_{i}}    \left\Vert f \chi_{Q_{i}} \right\Vert^{p}_{q}   d\mu(y) \right]^{\frac{1}{p}} \\
	& \leq & C(q,p) \left[  \sum\limits_{i \in \mathbb{Z} } \mu (Q_{i} )  \left\Vert f \chi_{Q_{i}} \right\Vert^{p}_{q} \right]^{\frac{1}{p}} = C(q,p) \widetilde{\left\|f\right\|}_{q,\, p}.
\end{eqnarray*}	
Applying again Lemma \ref{lem 20} and Proposition \ref{p2} we have
 \begin{eqnarray*}
\widetilde{\left\|f\right\|}_{q,p} &=& \left[  \sum\limits_{i \in \mathbb{Z} } \displaystyle \int_{Q_{i}} \left\Vert f \chi_{Q_{i}} \right\Vert^{p}_{q} d\mu(y) \right]^{\frac{1}{p}}\\
&\leq &\left[  \sum\limits_{i \in \mathbb{Z} } \displaystyle \int_{Q_{i}} \left\Vert f \chi_{I(y,1)} \right\Vert^{p}_{q}   d\mu(y) \right]^{\frac{1}{p}} =  \left[
\displaystyle \int_\mathbb{R}  \left\Vert f \chi_{I(y,1)} \right\Vert^{p}_{q} d\mu(y) \right]^{\frac{1}{p}}\\
& \leq & C(q)\left[
\displaystyle \int_\mathbb{R} \left[ \displaystyle \int_\mathbb{R}  \tau_{y} \vert f \vert^{q} (x) \chi_{B(0,1)} (x)  d\mu (x) \right]^{\frac{p}{q}} d\mu(y) \right]^{\frac{1}{p}}.
\end{eqnarray*}
Hence,  $$  \widetilde{\left\|f\right\|}_{q,p} \leq  C(q) \left\|f\right\|_{q,p}.$$	
$\bullet$ Suppose $p= \infty$. On the one hand,
according to Proposition \ref{p2}, there exists a constant $C > 0$ not depending on $f$ such that
\begin{eqnarray*}
\displaystyle \int_{B(0,1)}  \tau_{y} \vert f \vert^{q} (x)  d\mu (x) &\leq& C \displaystyle \int_{I(y,1)} \vert f \vert^{q} (x)  d\mu (x) \\
&\leq& C  \sum\limits_{l \in \mathbb{Z}: Q_{l}\cap I(y,1) \ne \emptyset }  \displaystyle \int_{Q_{l}}  \vert f \vert^{q} (x) \chi_{I(y,1)} (x)  d\mu (x) \\
&\leq& C \sup\limits_{l \in \mathbb{Z}} \left\Vert f \chi_{Q_{l}} \right\Vert^{q}_{q} \sum\limits_{l \in \mathbb{Z}: Q_{l}\cap I(y,1) \ne \emptyset }  1.
\end{eqnarray*}
Hence, $\quad \left\|f\right\|_{q,\infty} \leq C(q) \widetilde{\left\|f\right\|}_{q, \infty}.$\\
On the other hand, by applying again Lemma \ref{lem 20}, we have
$$\Vert f \chi_{Q_{l}} \Vert_{q} \leq \Vert f \chi_{I(y,1)} \Vert_{q},\;\;\;\;\; y \in Q_{l}.$$
Moreover, there exists a constant $C > 0$ not depending on $f$ such that
$$\displaystyle \int_{I(y,1)} \vert f \vert^{q} (x)  d\mu (x)\leq C\displaystyle \int_{B(0,1)}  \tau_{y} \vert f \vert^{q} (x)  d\mu (x).$$
Hence, $\quad \widetilde{\left\|f\right\|}_{q,\infty} \leq C(q) \left\|f\right\|_{q, \infty}.$\\
$2^{nd}$ case. We suppose that $q =  \infty$.\\
$\bullet$ Suppose $p< \infty$.
Thanks to Lemma 6.2 in \cite{NF} we have
	\begin{eqnarray*}
	\left\|f\right\|_{\infty,p}
	&\leq& \left\|\left\|f\chi_{I(\cdot,3)}\right\|_{\infty}\right\|_{p} + \left\|\left\|f\chi_{I(-\cdot,3)}\right\|_{\infty}\right\|_{p} = J_{1}+ J_{2}.
	\end{eqnarray*}
Notice that
$$ \vert f(x) \chi_{I(y,3)}(x) \vert \leq   \sum\limits_{l \in \mathbb{Z} }  \left\vert f(x) \chi_{I(y,3)}(x) \chi_{Q_{l}}(x)\right\vert , \;\;\;\;\; x \in \mathbb{R}.$$
Since for every integer $l$,
 $$I(x,3) \subset 7Q_{l}, \;\;\;\;\; x \in Q_{l},$$
we get $$\Vert f \chi_{I(y,3)} \Vert_{\infty} \leq  \sum\limits_{l \in \mathbb{Z}: Q_{l}\cap I(y,3) \ne \emptyset}  \left\Vert f \chi_{Q_{l}}\right\Vert_{\infty}\chi_{7Q_{l}}(y).$$
Thus, $$J_{1} \leq \left[
			\displaystyle \int_\mathbb{R} \left( \sum\limits_{l \in \mathbb{Z}: Q_{l}\cap I(y,3) \ne \emptyset}  \left\Vert f \chi_{Q_{l}}\right\Vert_{\infty}\chi_{7Q_{l}}(y) \right)^{p} d\mu(y) \right]^{\frac{1}{p}}.$$
As for all $y \in \mathbb{R}$ the set $ \left\{
			\begin{array}{ll}
			l  \in \mathbb{Z} : Q_{l}\cap I(y,3) \ne \emptyset
			\end{array}
			\right\}$ has at most seven elements, then by applying H\"older inequality we have
\begin{eqnarray*}
	J_{1} &\leq& 7^{1-\frac{1}{p}}\left[
			\displaystyle \int_\mathbb{R}  \sum\limits_{l \in \mathbb{Z}}  \left\Vert f \chi_{Q_{l}}\right\Vert^{p}_{\infty}\chi_{7Q_{l}}(y)  d\mu(y) \right]^{\frac{1}{p}}\\
		  &\leq& 7^{1-\frac{1}{p}} \left[ \sum\limits_{i \in \mathbb{Z}}
		  	\displaystyle \int_{Q_{i}} \sum\limits_{l \in \mathbb{Z}}  \left\Vert f \chi_{Q_{l}} \right\Vert^{p}_{\infty} \chi_{7Q_{l}} (y)  d\mu(y) \right]^{\frac{1}{p}}\\
		  	 &\leq& 7^{1-\frac{1}{p}} \left[ \sum\limits_{l \in \mathbb{Z}} \mu(7Q_{l})  \left\Vert f \chi_{Q_{l}} \right\Vert^{p}_{\infty} \right]^{\frac{1}{p}} . 		
	\end{eqnarray*}
Since $\mu$ is doubling, we get $ J_{1} \leq  C(p) 	\widetilde{\left\|f\right\|}_{\infty,p}.$ \\ Also, $ J_{2} \leq  C(p) 	\widetilde{\left\|f\right\|}_{\infty,p}$ because $\mu$ is a symmetric invariant measure. Thus the result follows. \\
For the reverse inequality, we apply Lemma \ref{lem 20} to get
$$ \sum\limits_{l \in \mathbb{Z}} \displaystyle \int_{Q_{l}}  \left\Vert f \chi_{Q_{l}} \right\Vert^{p}_{\infty} d\mu(y)\leq \sum\limits_{l \in \mathbb{Z}} \displaystyle \int_{Q_{l}}  \left\Vert f \chi_{I(y,1)} \right\Vert^{p}_{\infty} d\mu(y).$$
Then the previous inequality becomes
\begin{eqnarray*}
 & & \sum\limits_{l \in \mathbb{Z}} \displaystyle \int_{Q_{l}}  \left\Vert f \chi_{Q_{l}} \right\Vert^{p}_{\infty} d\mu(y) \\
& \leq & \displaystyle \int_{B(0,1)} \left\Vert f \chi_{I(y,1)} \right\Vert^{p}_{\infty} d\mu(y) +  \displaystyle \int_{B(0,1)^{c}}  \left\Vert f \chi_{I(y,1)} \right\Vert^{p}_{\infty} d\mu(y).
\end{eqnarray*}
Since $I(y,1) \subset B(y,1)$ for $\vert y \vert \ge 1$, we deduce that
\begin{eqnarray*}
\displaystyle \int_{B(0,1)^{c}}  \left\Vert f \chi_{I(y,1)} \right\Vert^{p}_{\infty} d\mu(y)
& \leq &  \displaystyle \int_{B(0,1)^{c}}  \left\Vert f \chi_{B(y,1)} \right\Vert^{p}_{\infty} d\mu(y) \\
& \leq & \displaystyle \int_\mathbb{R}  \left\Vert f \chi_{B(y,1)} \right\Vert^{p}_{\infty} d\mu(y)= \left\|f\right\|_{\infty,p}^{p}.
\end{eqnarray*}
As $I(y,1)\setminus \{0 \}  \subset B(y,1)$ when $\vert y \vert < 1$, then
\begin{eqnarray*}
\displaystyle \int_{B(0,1)} \left\Vert f \chi_{I(y,1)} \right\Vert^{p}_{\infty} d\mu(y) & = & \displaystyle \int_{B(0,1)} \left\Vert f \chi_{I(y,1)\setminus \{0 \}} \right\Vert^{p}_{\infty} d\mu(y) \\
& \leq & \displaystyle \int_\mathbb{R}  \left\Vert f \chi_{B(y,1)} \right\Vert^{p}_{\infty} d\mu(y)= \left\|f\right\|_{\infty,p}^{p}.
\end{eqnarray*}
Hence, $\quad \widetilde{\left\|f\right\|}_{\infty,p} \leq 2^{\frac{1}{p}} \left\|f\right\|_{\infty,p}.$ \\
$\bullet$ Suppose $p= \infty$. \\
By Proposition \ref{prop 3}, $\left\|f\right\|_{\infty,\infty} = \left\|f\right\|_{\infty}$.
To get the desired result, we will show that $\widetilde{\left\|f\right\|}_{\infty,\infty} = \left\|f\right\|_{\infty}$. \\
Since $$ \Vert f \chi_{Q_{l}} \Vert_{\infty} \leq  \left\|f\right\|_{\infty},\,\;\;\;\;\;\; l \in \mathbb{Z},$$ it follows that $\widetilde{\left\|f\right\|}_{\infty,\infty} \leq  \left\|f\right\|_{\infty}.$ \\ For the converse, we assume that $ \Vert f \Vert_{ \infty} > 0$ , otherwise the result is obvious.\\
Let $r$ be a real number satysftying $0<r< \left\|f\right\|_{\infty}$. Then
$$ 0<\mu \left(\left\{ x \in \mathbb{R} \, :  | f(x) | > r \right\} \right) = \mu \left( \bigcup_{l \in \mathbb{Z}}  \left\{ x \in  Q_{l} \; : \;  | f(x) | > r \right\} \right).$$
Thus, there exists an integer $l$ satisfying $0 < \mu \left( \left\{x \in  Q_{l} \, : \;  | f(x) | > r \right\} \right).$ Hence, $r< \Vert f \chi_{Q_{l}} \Vert_{\infty} \leq \widetilde{\left\|f\right\|}_{\infty,\infty}$. Therefore,
$ \left\|f\right\|_{\infty} \leq \widetilde{\left\|f\right\|}_{\infty,\infty}$ and the claim follows.
\end{proof}
It is easy to see that the map $f\mapsto \widetilde{\| f \|}_{q, \, p}$ defines a norm on $(L^{p},\, L^{q})(\mu)$.
\begin{lem}\label{rem 5} There is a constant $C>0$ such that for all integers $l$ we have
$$ \mu(I(0,1))  \leq C \mu(Q_{l}).$$
\end{lem}
\begin{proof}
It is proved in \cite{NF} that $$\mu(I(0,1))  \leq 2 \mu(I(x,1)),\,\,\,\,\,\,\, x \in \mathbb{R}.$$
Thus, for $x \in Q_{l}$, we have by Lemma  \ref{lem 20} that
$$\mu(I(0,1))  \leq 2 \mu(I(x,1))\leq 2 \mu (3 Q_{l}).$$
The desired result follows from the fact that $\mu$ is doubling with respect to $Q_{l}$.
\end{proof}

\begin{prop} \label{prop 46}
	Let $1 \leq q, \, p_{1}, \, p_{2} \leq \infty$ and $p_{2}\geq p_{1}$. Then there exists a constant $C >0$ such that
$$\left\| f \right\|_{q, \, p_{2}}   \leq C \: \left\| f  \right\|_{q,\, p_{1}},\;\,\;\;\;\;\;\; f \in L^{0}(\mu).$$
\end{prop}
\begin{proof}
	Let $f \in L^{0}(\mu)$. If $p_{1} = p_{2}$ then the result is obvious. \\
	$\bullet$ We suppose that $p_{1} < p_{2}< \infty$.	
	  We have
	\begin{eqnarray*}
		\widetilde{\left\|f\right\|}_{q,\, p_{2}}  &=& \left[\sum\limits_{i \in \mathbb{Z}}  \left[ \mu (Q_{i})^{\frac{1}{p_{2}}} \left\Vert f \chi_{Q_{i}} \right\Vert_{q} \right]^{p_{2}} \right]^{\frac{1}{p_{2}}}\\
		& \leq & \left[\sum\limits_{i \in \mathbb{Z} }\mu (Q_{i} )^{\frac{p_{1}}{p_{2}}} \left\Vert f \chi_{Q_{i}} \right\Vert^{p_{1}}_{q} \right]^{\frac{1}{p_{1}}} \\
		&=& \left[\sum\limits_{i \in \mathbb{Z} }  \mu (Q_{i} )^{\frac{p_{1}}{p_{2}}-1} \mu (Q_{i} ) \left\Vert f \chi_{Q_{i}} \right\Vert^{p_{1}}_{q} \right]^{\frac{1}{p_{1}}}.
	\end{eqnarray*}	
An application of Lemma \ref{rem 5} and Proposition \ref{prop 45} leads to the desired result.\\
$\bullet$ Suppose $p_{1} < p_{2}=\infty$. 	
For all $l \in \mathbb{Z}$, $$\left\Vert f \chi_{Q_{l}} \right\Vert^{p_{1}}_{q} \leq \sum\limits_{l \in \mathbb{Z} } 1 \left\Vert f \chi_{Q_{l}} \right\Vert^{p_{1}}_{q} \leq C\mu (I(0,1))^{-1} \sum\limits_{l \in \mathbb{Z} } \mu(Q_{l} )  \left\Vert f \chi_{Q_{l}} \right\Vert^{p_{1}}_{q}$$
according to Lemma \ref{rem 5}. The result follows by Proposition \ref{prop 45}.		
\end{proof}
Proposition \ref{prop 45} and Proposition \ref{prop 46} lead to the following remark.
\begin{remark}
  For $1 \leq \, q,\, p_{1},\, p_{2},\ \leq \infty$ and $p_{2}\geq p_{1}$ we have
  $$(L^{q},\, L^{p_{1}})(\mu) \subset (L^{q},\, L^{p_{2}})(\mu).$$
   If in addition $p_{2}=q$, then by Proposition \ref{prop 3},
  $$ (L^{q},\, L^{p_{1}})(\mu) \subset L^{q}(\mu).$$
  \label{R2}
\end{remark}
For $1\leq q < \infty$, we denote by $L_{c}^{q} (\mu)$ the set of elements of $L^{q} (\mu)$ with compact support. For any element $f$ of $L^{0}(\mu)$, we will denote by $supp (f)$ its support.\\
The following result shows that $L_{c}^{q} (\mu) \subset (L^{q}, L^{p}) (\mu)$ when $1 \leq q,p < \infty$.
\begin{lem} \label{lem 21}
	Let $1 \leq q, \, p < \infty$. If $f \in L_{c}^{q} (\mu)$, then there is a constant $C>0$ such that
$$\widetilde{\left\|f\right\|}_{q,p} \leq C \left\Vert f  \right\Vert_{q}.$$
\end{lem}
\begin{proof}
Let $f \in L_{c}^{q} (\mu)$. Let us set $K=supp(f)$. We have
	\begin{eqnarray*}
		\widetilde{\left\|f\right\|}_{q,p} &=&  \left[\sum\limits_{l \in \mathbb{Z}: Q_{l}\cap K \ne \emptyset } \mu (Q_{l}) \left\Vert f \chi_{Q_{l}\cap K} \right\Vert^{p}_{q} \right]^{\frac{1}{p}}\\
		&\leq &  \left\Vert f \chi_{ K} \right\Vert_{q}\left[\sum\limits_{l \in \mathbb{Z}: Q_{l}\cap K \ne \emptyset } \mu (Q_{l})  \right]^{\frac{1}{p}}=C \|f \|_{q},
	\end{eqnarray*}
	since the set $\{l \in \mathbb{Z}: Q_{l}\cap K \ne \emptyset \}$ has a finite number of elements.
\end{proof}

\begin{prop} \label{prop 48}
	Let $1 \leq q, \, p < \infty$. Then $L_{c}^{q} (\mu)$ is a dense subspace of $(L^{q}, L^{p}) (\mu)$.
\end{prop}
\begin{proof}
Let $1 \leq q, \, p < \infty$ and $f \in (L^{q},\, L^{p}) (\mu)$. Let $n$ be a positive integer. We consider the function $f_n$ defined by $$f_{n} (x) = f(x) \max(1- \vert \frac{x}{n} \vert ,0), \,\,\,\,\,\,\, x \in \mathbb{R}.$$
a) We have $suppf_{n} \subset B(0,n)$.\\
We will show that $f_{n}$ belongs to $L^{q} (\mu)$ by examining two cases.\\
$\bullet$ Suppose that $p \leq q$.\\
We have
		\begin{eqnarray*}
			\widetilde{\left\|f_{n}\right\|}_{q, \, q}  & \leq &  \left[\sum\limits_{l \in \mathbb{Z} } \left[ (\mu (Q_{l}))^{\frac{1}{q}} \left\Vert f \chi_{Q_{l}} \right\Vert_{q} \right]^{q} \right]^{\frac{1}{q}}
			 \leq  \left[\sum\limits_{l \in \mathbb{Z} }  (\mu (Q_{l}))^{\frac{p}{q}} \left\Vert f\chi_{Q_{l}} \right\Vert^{p}_{q}  \right]^{\frac{1}{p}} \\ &\leq &\left[\sum\limits_{l \in \mathbb{Z} }  (\mu (Q_{l}))^{\frac{p}{q} -1} \mu (Q_{l}) \left\Vert f \chi_{Q_{l}} \right\Vert^{p}_{q}  \right]^{\frac{1}{p}}.
		\end{eqnarray*}
It follows that if $p=q$ then $\widetilde{\left\|f_{n}\right\|}_{q, \, q} \leq \widetilde{\left\|f \right\|}_{q, \, p} $. If $p < q$ then an application of Lemma \ref{rem 5} ensures the existence of a constant $C> 0$ such that $\widetilde{\left\|f_{n}\right\|}_{q, \, q} \leq C  \widetilde{\left\|f\right\|}_{q, \; p}$.\\
We deduce from Proposition \ref{prop 3} that $f \in L^{q}(\mu)$.\\
$\bullet$ Suppose that $q < p$.\\
We have
		\begin{eqnarray*}
			\widetilde{\left\|f_{n}\right\|}^{q}_{q,\, q}
			&\leq & \sum\limits_{l \in \mathbb{Z}\,: \, Q_{l}\cap B(0,n) \ne \emptyset } \mu (Q_{l}) \displaystyle \int_{B(0,n)}  \left\vert f(x) \chi_{Q_{l}} (x) \right\vert^{q} d\mu(x) \\
			&\leq & \sum\limits_{l \in \mathbb{Z} \, : \, Q_{l}\cap B(0,n) \ne \emptyset } (\mu (Q_{l}))^{\frac{1}{s} + \frac{1}{s^{'}}} \displaystyle \int_{B(0,n)}  \left\vert f(x) \chi_{Q_{l}} (x) \right\vert^{q} d\mu(x),
		\end{eqnarray*}
		where $s= \frac{p}{q}$ and $s^{\prime}= \frac{p}{p-q}$. Then, it follows from H\"older inequality that
		$$ \widetilde{\left\|f_{n}\right\|}^{q}_{q,\, q} \leq \left[ \sum\limits_{l \in \mathbb{Z}: Q_{l}\cap B(0,n) \ne \emptyset } \mu (Q_{l})  \right]^{\frac{1}{s^{\prime}}} \left[\sum\limits_{l \in \mathbb{Z}} \mu (Q_{l}) \left\Vert f \chi_{Q_{l}} \right\Vert^{p}_{q} \right]^{\frac{q}{p}}.$$ Since the set $\{
		l  \in \mathbb{Z} \, : \,  Q_{l}\cap B(0,n) \ne \emptyset \}$ has
$2n$ elements, we have that
$$D_{n} = \left[ \sum\limits_{l \in \mathbb{Z}: Q_{l}\cap B(0,n) \ne \emptyset } \mu (Q_{l})  \right]^{\frac{1}{s^{\prime}q}}$$
is a positive real number satisfying $\widetilde{\left\|f_{n}\right\|}_{q,q} \leq D_{n}  \widetilde{\left\|f\right\|}_{q,\, p}$. It follows again from Proposition \ref{prop 3} that $f_{n} \in L^{q}(\mu)$.\\
b) We have,
\begin{eqnarray*}
\left\|f_{n} - f \right\|_{q,p}&= & \left\|  \left\|(f_{n}-f) \tau_{-y}^{\frac{1}{q}} \chi_{B(0,1)} \right\|_{q} \right\|_{p}\\
& =  & \left\|  \left\|f_{n} \tau_{-y}^{\frac{1}{q}} \chi_{B(0,1)} - f \tau_{-y}^{\frac{1}{q}} \chi_{B(0,1)} \right\|_{q} \right\|_{p},
\end{eqnarray*}
where $ \left\|  \left\|(f_{n}-f) \tau_{-y}^{\frac{1}{q}} \chi_{B(0,1)} \right\|_{q} \right\|_{p} $ denotes the $L^{p}(\mu)$-norm of the function
$y\mapsto \left\|  \left\|(f_{n}-f) \tau_{-y}^{\frac{1}{q}} \chi_{B(0,1)} \right\|_{q} \right\|_{p} $.\\
	Let $y \in \mathbb{R}$. Then $(f_{n} \tau_{-y}^{\frac{1}{q}} \chi_{B(0,1)})_{n \ge 1}$ is a $\mu$-measurable sequence of function which converges to $f \tau_{-y}^{\frac{1}{q}} \chi_{B(0,1)}$ $\mu$-almost everywhere. In addition, for $\mu$-almost every $y \in \mathbb{R}$,
	$$\left\vert f_{n} \tau_{-y}^{\frac{1}{q}} \chi_{B(0,1)} \right\vert  \leq \left\vert f \tau_{-y}^{\frac{1}{q}} \chi_{B(0,1)} \right\vert \in L^{q} (\mu) $$ because $f \in (L^{q}, L^{p}) (\mu)$. Hence, applying the dominated convergence theorem we obtain
$$\lim\limits_{n \to \infty} \left\|f_{n} \tau_{-y}^{\frac{1}{q}} \chi_{B(0,1)} - f \tau_{-y}^{\frac{1}{q}} \chi_{B(0,1)} \right\|_{q} =0.$$
Moreover,
$$ \left\|f_{n} \tau_{-y}^{\frac{1}{q}} \chi_{B(0,1)} - f \tau_{-y}^{\frac{1}{q}} \chi_{B(0,1)} \right\|_{q} \leq 2 \left\| f \tau_{-y}^{\frac{1}{q}} \chi_{B(0,1)} \right\|_{q} $$
 and the function $y\mapsto \left\| f \tau_{-y}^{\frac{1}{q}} \chi_{B(0,1)} \right\|_{q} $ belongs to $L^{p}(\mu)$ since $f \in (L^{q}, L^{p})(\mu)$. Hence, applying again the dominated convergence theorem we obtain
	$$ \lim\limits_{n \to \infty}  \left\|  \left\|f_{n} \tau_{-y}^{\frac{1}{q}} \chi_{B(0,1)} - f \tau_{-y}^{\frac{1}{q}} \chi_{B(0,1)} \right\|_{q} \right\|_{p} = 0.$$
From parts a) and b), we conclude that $L_{c}^{q} (\mu)$ is a dense subspace of $(L^{q}, \,  L^{p}) (\mu)$.	
\end{proof}
Let $1\leq q, p\leq\infty$. For all $r>0$, we define
$$
(L^{q}, L^{p})_{r} (\mu) = \left\{
\begin{array}{ll}
f \in L^{0}(\mu) : \quad  _{r}\Vert f \Vert_{q,p} < \infty
\end{array}
\right\}
$$
and for any $\mu$-measurable function $f$ on $\mathbb{R}$, we set
\begin{equation*}
	_{r}	\widetilde{\left\|f\right\|}_{q,p}=\left\{\begin{array}{lll}\left(\sum\limits_{l \in \mathbb{Z} } \mu (Q_{l}^{r} ) \left\Vert f \chi_{Q_{l}^{r}} \right\Vert^{p}_{q}\right)^{\frac{1}{p}}&\text{ if }&p<\infty , \\
			\sup\limits_{l \in \mathbb{Z}} \left\Vert f \chi_{Q_{l}^{r}} \right\Vert_{q}&\text{ si }&p=\infty,
		\end{array}\right.
		\end{equation*}
where $Q_{l}^{r}=[rl,\, rl+r)$.	\\
It is easy to see that an adaptation of the proofs of the propositions stated in \cite[Section 3]{NF} and those given in this section (by replacing $\mu(B_{1})$ and $Q_{l}$ by
$\mu(B_r)$ and $Q_{l}^r$ respectively) yields the following result.
\begin{prop}\label{pro-n}
Let $1\leq q, p\leq\infty$ and $r>0$. Then the following assertions hold.
\begin{enumerate}
\item $(L^{q}, L^{p})_{r} (\mu)$ is a complex subspace of the space $L^{0}(\mu)$.
\item The maps $f\mapsto \,  _{r}\Vert f \Vert_{q,p}$ and $f \mapsto \,  _{r}\widetilde{\left\|f\right\|}_{q,p}$ define norms on $(L^{q}, L^{p})_{r}(\mu)$.
\item The norms $  _{r}\Vert f \Vert_{q,p} $ and $ _{r}\widetilde{\left\|f\right\|}_{q,p} $ are equivalent.
\item $((L^{q}, L^{p})_{r} (\mu), \,  _{r}\Vert f \Vert_{q,p}) $ is a complex Banach space.
\item If $1\leq p_{1}\leq p_{2}\leq \infty$ then there exists a constant $C>0$ not depending on $r$ such that
$$_{r}\left\|f\right\|_{q,p_{2}} \leq C \mu(B(0,r))^{\frac{1}{p_{2}} - \frac{1}{p_{1}}} \: _{r}\left\|f\right\|_{q,p_{1}},\;\;\;\;\; f \in L^{0}(\mu).$$
\end{enumerate}
\end{prop}
Now, we state some new properties of the spaces $\left(L^{q},L^{p}\right)^{\alpha}(\mu)$.\\
Using the same arguments as in the proof of Proposition \ref{prop 4} with $B_r$ instead of $B_1$,  it is easy to see  that H\"older inequality remains valid in the Dunkl-Fofana spaces. More precisely, we have the following
\begin{prop}	
	Let $(q_{1} , p_{1}, \alpha_{1} )$ and $(q_{2} , p_{2}, \alpha_{2})$ two elements of $[1, +\infty]^{3}$ such that $q_{1} \leq \alpha_{1} \leq p_{1}$ and $q_{2} \leq \alpha_{2} \leq p_{2}$. \\ If
	\begin{center}
		$\frac{1}{q_{1}} + \frac{1}{q_{2}} = \frac{1}{q} \leq 1$, \: $\frac{1}{p_{1}} + \frac{1}{p_{2}} = \frac{1}{p} \leq 1$ and $\frac{1}{\alpha_{1}} + \frac{1}{\alpha_{2}} = \frac{1}{\alpha}$,
	\end{center}
	then for all $f$ and $g$ of $L^{0} (\mu)$, we have
	$$\Vert f g\Vert_{q,p,\alpha} \leq \Vert f \Vert_{q_{1} , p_{1}, \alpha_{1} } \Vert g \Vert_{q_{2} , p_{2}, \alpha_{2}} .$$
\end{prop}
From Proposition \ref{prop 38} we deduce the following result.
\begin{prop}
 Let $1 \leq q <s \leq p<\infty$. Then
  $$\Vert f \Vert_{q,p,\alpha}  \leq C\Vert f \Vert_{ L^{\alpha, \infty}(\mu)} \quad \forall \: f \in  L^{0}(\mu),$$
where $C>0$ is a constant not depending on $f$.
\end{prop}
The above result shows that $L^{\alpha,\infty}(\mu)$ is continuously embedded in $(L^q,L^p)^\alpha(\mu)$.\\
The following proposition asserts that endowed with the norm $\Vert \cdot \Vert_{q,p,\alpha}$, $ (L^{q}, L^{p})^{\alpha} (\mu)$ is a complex Banach space.
\begin{prop}
	Let $1\leq \alpha ,\, p,\, q \leq \infty$ with $q \leq  \alpha  \leq p$. Then $ \left((L^{q}, L^{p})^{\alpha} (\mu),\: \Vert \cdot \Vert_{q,p,\alpha} \right) $ is a complex Banach space.
\end{prop}
\begin{proof}
	Let $ (f_{j})_{j \in \mathbb{N}} $ be a Cauchy sequence of elements of $ (L^{q}, L^{p})^{\alpha} (\mu)$. \\ Fix $\epsilon > 0$. Then there exists an integer $N$ such that for all integers $m$ and $n$ we have
	$$  m, \, n > N \implies \Vert f_{m} - f_{n} \Vert_{q,p,\alpha}  < \epsilon. $$
	This implies that for $r$ fixed and for $m,\, n>N$,
	\begin{center}
		$\mu(B(0,r))^{\frac{1}{\alpha} - \frac{1}{q} - \frac{1}{p}} \: _{r}\Vert f_{m} - f_{n} \Vert_{q,p}  < \epsilon. $
	\end{center}
	Hence, $ (f_{j})_{j \in \mathbb{N}} $ is a Cauchy sequence of $ (L^{q}, L^{p})_{r} (\mu)$. Since
	$ (L^{q}, L^{p})_{r} (\mu)$ is a complex Banach space, $ (f_{j})_{j \in \mathbb{N}} $ converges to $f \in (L^{q}, L^{p})_{r} (\mu)$. In addition, for all $m,n > N$ we have
		$$ \vert \Vert f_{m} \Vert_{q,p,\alpha} - \Vert f_{n} \Vert_{q,p,\alpha} \vert \leq \Vert f_{m} - f_{n} \Vert_{q,p,\alpha} < \epsilon.$$
		
	Therefore, $(\Vert f_{j} \Vert_{q,p,\alpha})_{j \in \mathbb{N}}$ is a Cauchy sequence on $\mathbb{R}$. Let us note $M \ge 0$ its limit. It follows that for $m>N$,	
	\begin{eqnarray*}
		\mu (B(0,r))^{\frac{1}{\alpha} - \frac{1}{q} - \frac{1}{p}} \:_{r}\Vert f \Vert_{q,p}
		& \leq & \Vert f_{m} \Vert_{q,p,\alpha} + \mu (B(0,r))^{\frac{1}{\alpha} - \frac{1}{q} - \frac{1}{p}} \:_{r}\Vert f - f_{m} \Vert_{q,p}.
	\end{eqnarray*}
	By letting $m$ go to $\infty$, we get
		$$\mu (B(0,r))^{\frac{1}{\alpha} - \frac{1}{q} - \frac{1}{p}} \:_{r}\Vert f \Vert_{q,p} \leq M.$$ \\ Hence, $\Vert f \Vert_{q,p,\alpha} \leq M $ and $f \in  (L^{q}, L^{p})^{\alpha} (\mu)$. \\
	Let $m$ be an integer such that $m>N$. For all $r>0 $ and $l \in \mathbb{N}$, we have	
	\begin{eqnarray*}
		& & \mu (B(0,r))^{\frac{1}{\alpha} - \frac{1}{q} - \frac{1}{p}} \:_{r}\Vert f_{m} -f \Vert_{q,p} \\
		& \leq & \Vert f_{m} - f_{m+l}  \Vert_{q,p,\alpha} + \mu (B(0,r))^{\frac{1}{\alpha} - \frac{1}{q} - \frac{1}{p}} \:_{r}\Vert f_{m+l} -f \Vert_{q,p}\\
		& \leq & \epsilon + \mu(B(0,r))^{\frac{1}{\alpha} - \frac{1}{q} - \frac{1}{p}} \:_{r}\Vert f_{m+l} -f \Vert_{q,p}.
	\end{eqnarray*}
Since	
$$ \lim\limits_{l \rightarrow	  \infty}  \mu (B(0,r))^{\frac{1}{\alpha} - \frac{1}{q} - \frac{1}{p}} \:_{r}\Vert f_{m+l} -f \Vert_{q,p}  = 0,$$
we have
	 $$ \Vert f_{m} - f \Vert_{q,p,\alpha}  \leq \epsilon .$$
This ends the proof.	
\end{proof}
The family of spaces $\left(L^{q},L^{p}\right)^{\alpha}(\mu)$ is increasing with respect to the $p$ power as stated below.
\begin{prop}
Let $ 1 \leq q \leq \alpha \leq p_{1} \leq p_{2} \leq  \infty$. There exists a constant $C>0$ such that
$$ \Vert f \Vert_{q,p_{2},\alpha} \leq C \Vert f \Vert_{q,p_{1}, \alpha},\;\;\;\;\;\;f \in L^{0}(\mu), $$
and consequently, $\left(L^{q},L^{p_{1}}\right)^{\alpha}(\mu) \subset  \left(L^{q},L^{p_{2}}\right)^{\alpha}(\mu)$.
\end{prop}
\begin{proof}
Let $f \in L^{0}(\mu)$. Since for $p_{1} = p_{2}$ the result is obvious, we give the proof for $p_{1} < p_{2}$.\\
Let $r>0$.\\
$\bullet$ Suppose that $p_{2}< \infty$.\\
According to Proposition \ref{pro-n}, there exists a constant $C>0$, not depending on $r$ and $f$, such that
$$_{r}\left\|f\right\|_{q,p_{2}} \leq C \mu(B(0,r))^{\frac{1}{p_{2}} - \frac{1}{p_{1}}} \: _{r}\left\|f\right\|_{q,p_{1}}.$$
This implies that
$$\mu(B(0,r))^{-\frac{1}{p_{2}}}\: _{r}\left\|f\right\|_{q,p_{2}} \leq C \mu(B(0,r))^{ - \frac{1}{p_{1}}} \: _{r}\left\|f\right\|_{q,p_{1}}. $$
Multiplying the both sides of the above inequality by $(\mu_{k} (I(0,r)))^{\frac{1}{\alpha} -\frac{1}{q}}$, we get
$$\mu(B(0,r))^{\frac{1}{\alpha} -\frac{1}{q} -\frac{1}{p_{2}}}\: _{r}\left\|f\right\|_{q,p_{2}} \leq C \mu(B(0,r))^{ \frac{1}{\alpha} -\frac{1}{q} - \frac{1}{p_{1}}} \: _{r}\left\|f\right\|_{q,p_{1}}. $$
Hence $ \quad \Vert f \Vert_{q,p_{2},\alpha} \leq C \Vert f \Vert_{q,p_{1}, \alpha}$. \\
$\bullet$ Suppose that $p_{2}=\infty$.\\
By Proposition \ref{pro-n}, there exists a constant $C>0$, not depending on $r$ and $f$, such that
$$  _{r}\Vert f \Vert_{q, \infty} \leq C (\mu (B(0,r)))^{-\frac{1}{p_{1}}} \: _{r}\Vert f \Vert_{q,p_{1}}.$$
This implies that
 $$\mu(B(0,r))^{\frac{1}{\alpha} -\frac{1}{q}}\: _{r}\left\|f\right\|_{q,\infty} \leq C \mu(B(0,r))^{ \frac{1}{\alpha} -\frac{1}{q} - \frac{1}{p_{1}}} \: _{r}\left\|f\right\|_{q,p_{1}}. $$
Hence $ \quad \Vert f \Vert_{q, \infty,\alpha} \leq C \Vert f \Vert_{q,p_{1}, \alpha}$. \\
We conclude that $\left(L^{q},L^{p_{1}}\right)^{\alpha}(\mu) \subset  \left(L^{q},L^{p_{2}}\right)^{\alpha}(\mu)$.
\end{proof}
\section{Proof of the main results}
For $1< q < \infty $, we will denote by $q'$ its conjugate exponent: $\frac{1}{q}+ \frac{1}{q'}=1$. We will need the following results.
\begin{prop}
\label{maxi}(\cite{NF})
Let $1< q\leq \alpha\leq p\leq \infty$.
There exists a constant $C>0$ such that
$$\left\|Mf\right\|_{q,p,\alpha}\leq C\left\|f\right\|_{q,p,\alpha}, \ \ f\in L_{loc}^{1}(\mu).$$
\end{prop}
\begin{prop}\label{weakmaxi}(\cite{NF})
Let $1\leq \alpha\leq p\leq \infty$. There exists  $C>0$ such that
$$ \Vert M f \Vert_{(L^{1, \infty}, L^{p})^{\alpha} (\mu)} \leq C \Vert f \Vert_{1,p,\alpha}, \ \ f\in L_{loc}^{1}(\mu).$$
\end{prop}
\begin{proof}[Proof of Theorem \ref{theo 01}]
Let  $1 < q \leq \alpha \leq p < \infty$, $0<\beta< \frac{2k+2}{\alpha}$ and $f \in (L^{q}, L^{p})^{\alpha} (\mu)$.
We assume that  $f\ne 0$ $\mu$-almost everywhere, otherwise the result is obvious.\\
Let $r>0$ and $x \in \mathbb{R}$. Then
\begin{eqnarray*}
 I_{\beta} f(x)  & = & \displaystyle \int_{B(0,r)} \tau_{x} f (z) \vert z \vert^{\beta - 2k-2} d\mu (z) + \\
  & & \displaystyle \int_{B(0,r)^{c}} \tau_{x} f (z) \vert z \vert^{\beta - 2k-2} d\mu (z) = A(x) +B(x).
\end{eqnarray*}
Following the ideas of part (1) of the proof of Theorem 4 in \cite{GM}, we have
\begin{eqnarray*}
\vert A(x) \vert
 &=& \left\vert \sum^{\infty}_{i=0} \displaystyle \int_{B(0,2^{-i}r)\setminus{B(0,2^{-i-1}r)}} \tau_{x} f (z)  \vert z \vert^{\beta - 2k-2} d\mu (z) \right\vert \\
 &=& \left\vert \sum^{\infty}_{i=0} \displaystyle \int_{\mathbb{R}} f (z) \tau_{-x}(\vert \cdot \vert^{\beta - 2k-2} \chi_{B(0,2^{-i}r)\setminus{B(0,2^{-i-1}r)}})(z) d\mu(z) \right\vert \\
 &\leq& \sum^{\infty}_{i=0} \displaystyle \int_{\mathbb{R}}  \vert f (z) \vert \tau_{-x}(\vert \cdot \vert^{\beta - 2k-2} \chi_{B(0,2^{-i}r)\setminus{B(0,2^{-i-1}r)}})(z) d\mu (z)\\
 & & \text{( according to \cite[Theorem 6.3.4]{DX}).}
\end{eqnarray*}
It follows that
\begin{eqnarray*}
\vert A(x) \vert &\leq& \sum^{\infty}_{i=0} \displaystyle \int_{B(0,2^{-i}r)\setminus{B(0,2^{-i-1}r)}} \tau_{x} \vert f (z) \vert \vert z \vert^{\beta - 2k-2} d\mu(z) \\
&\leq& \sum^{\infty}_{i=0} (2^{-i-1}r)^{\beta - 2k-2} d_{k}(2^{-i}r)^{2k+2} M f(x) \\
&\leq& d_{k}2^{2k+2-\beta} r^{\beta} M f(x)\sum^{\infty}_{i=0} (2^{-1})^{\beta}.
\end{eqnarray*}
Hence,
\begin{eqnarray} \label{rel 4}
\vert A(x) \vert \leq C r^{\beta} M f(x).
\end{eqnarray}
On the other hand, we have
 \begin{eqnarray*}
& & \vert B(x) \vert \\
 &=& \left\vert \sum^{\infty}_{i=0} \displaystyle \int_{B(0,2^{i+1}r)\setminus{B(0,2^{i}r)}} \tau_{x} f (z)  \vert z \vert^{\beta - 2k-2} d\mu (z) \right\vert\\
 &=&\left\vert \sum^{\infty}_{i=0} \displaystyle \int_{\mathbb{R}} f (z) \tau_{-x}(\vert \cdot \vert^{\beta - 2k-2} \chi_{B(0,2^{i+1}r)\setminus{B(0,2^{i}r)}})(z) d\mu (z) \right\vert.
 \end{eqnarray*}
 By H\"older inequality we have
 \begin{eqnarray*}
& & \left\vert \sum^{\infty}_{i=0} \displaystyle \int_{\mathbb{R}} f (z) \tau_{-x}(\vert \cdot \vert^{\beta - 2k-2} \chi_{B(0,2^{i+1}r)\setminus{B(0,2^{i}r)}})(z) d\mu (z) \right\vert\\
 &\leq& \sum^{\infty}_{i=0} \displaystyle \int_{\mathbb{R}}  \vert f (z) \vert \tau_{-x}(\vert \cdot \vert^{\beta - 2k-2} \chi_{B(0,2^{i+1}r)\setminus{B(0,2^{i}r)}})^{\frac{1}{q}+\frac{1}{q^{'}}}(z) d\mu (z) \\
   &\leq& \sum^{\infty}_{i=0} \left(  \displaystyle \int_{\mathbb{R}}  \vert f (z) \vert^{q} \tau_{-x}(\vert \cdot \vert^{\beta - 2k-2} \chi_{B(0,2^{i+1}r)\setminus{B(0,2^{i}r)}})(z) d\mu (z)         \right)^{\frac{1}{q}}\\ &\times& \left( \displaystyle \int_{\mathbb{R}} \tau_{-x}(\vert \cdot \vert^{\beta - 2k-2} \chi_{B(0,2^{i+1}r)\setminus{B(0,2^{i}r)}})(z) d\mu (z) \right)^{\frac{1}{q^{'}}}.
 \end{eqnarray*}
So, thanks to \cite[Theorem 6.3.4]{DX}, we write that
\begin{eqnarray*}
 \vert B(x) \vert
 &\leq& \sum\limits^{\infty}_{i=0} \left(  \displaystyle \int_{2^i  r\leq \vert z\vert <2^{i+1} r} \tau_{x} \vert f (z) \vert^{q} \vert z \vert^{\beta - 2k-2} d\mu (z) \right)^{\frac{1}{q}} \times \\
 & & \left( \displaystyle \int_{2^i  r\leq \vert z \vert <2^{i+1} r} \vert z \vert^{\beta - 2k-2} d\mu (z) \right)^{\frac{1}{q^{'}}} \\
 &\leq& \sum\limits^{\infty}_{i=0} (2^{i}r)^{\frac{\beta - 2k-2}{q}}\left(  \displaystyle \int_{2^i  r\leq \vert z \vert <2^{i+1} r} \tau_{x} \vert f (z) \vert^{q} d\mu (z) \right)^{\frac{1}{q}} \times \\
 & & \left( (2^{i}r)^{\beta - 2k-2} \mu (B(0,2^{i+1}r)) \right)^{\frac{1}{q^{'}}}.
\end{eqnarray*}
As
$$ \left(  \displaystyle \int_{2^i  r\leq \vert z \vert <2^{i+1} r} \tau_{x} \vert f (z) \vert^{q} d\mu (z) \right)^{\frac{1}{q}}\leq \left\Vert  \left(  \displaystyle \int_{B(0,2^{i+1}r)} \tau_{x} \vert f (z) \vert^{q} d\mu (z) \right)^{\frac{1}{q}} \right\Vert_{\infty} ,$$
it follows that
\begin{eqnarray*}
     \vert B(x) \vert
&\leq& \sum\limits^{\infty}_{i=0} (2^{i}r)^{\beta - 2k-2} \mu (B(0,2^{i+1}r))^{\frac{1}{q^{'}}} \mu (B(0,2^{i+1}r))^{-\frac{1}{\alpha} + \frac{1}{q}} \\
 &\times& \mu (B(0,2^{i+1}r))^{\frac{1}{\alpha} - \frac{1}{q}} \left\Vert  \left(  \displaystyle \int_{B(0,2^{i+1}r)} \tau_{x} \vert f (z) \vert^{q} d\mu (z) \right)^{\frac{1}{q}} \right\Vert_{\infty} \\
&\leq& \Vert f \Vert_{q,\infty, \alpha} \sum\limits^{\infty}_{i=0} (2^{i}r)^{\beta - 2k-2} \mu (B(0,2^{i+1}r))^{1-\frac{1}{\alpha}}.
\end{eqnarray*}
Since $$\mu (B(0,2^{i+1}r))^{1-\frac{1}{\alpha}} = d^{1-\frac{1}{\alpha}}_{k} (2^{i+1}r)^{(2k+2)(1-\frac{1}{\alpha})}, $$
we have
\begin{eqnarray*}
& & \sum\limits^{\infty}_{i=0} (2^{i}r)^{\beta - 2k-2} \mu (B(0,2^{i+1}r))^{1-\frac{1}{\alpha}}\\
& =  & d^{1-\frac{1}{\alpha}}_{k}2^{(2k+2)(1-\frac{1}{\alpha})} r^{\beta -\frac{2k+2}{\alpha}} \sum\limits^{\infty}_{i=0} (2^{i})^{\beta - \frac{2k+2}{\alpha}}.
\end{eqnarray*}
In addition, $\sum\limits^{\infty}_{i=0} (2^{i})^{\beta - \frac{2k+2}{\alpha}}< \infty $. So
\begin{eqnarray}\label{rel 5}
\vert B(x) \vert \leq C r^{\beta -\frac{2k+2}{\alpha}} \Vert f \Vert_{q, \infty, \alpha}.
\end{eqnarray}
It follows from (\ref{rel 4}) and (\ref{rel 5}) that
$$ \vert I_{\beta} f(x) \vert \leq C ( r^{\beta}M f(x)  + r^{\beta -\frac{2k+2}{\alpha}} \Vert f \Vert_{q,\infty, \alpha}  )$$
for all $r>0$. Thus, taking $r= \left( \frac{\Vert f \Vert_{q, \infty, \alpha} }{M f(x)}   \right)^{\frac{\alpha}{2k+2}}$, we get
\begin{eqnarray} \label{rel 7}
\vert I_{\beta} f(x) \vert \leq C (M f(x))^{1- \frac{\alpha\beta}{2k+2}}\Vert f \Vert^{\frac{\alpha\beta}{2k+2}}_{q,\infty, \alpha}
\end{eqnarray}
for all $ x \in \mathbb{R}.$\\
From (\ref{rel 7}), we have
\begin{eqnarray*}
_{r}\Vert I_{\beta} f \Vert_{\bar{q},\bar{p}} &\leq& C \Vert f \Vert^{\frac{\alpha\beta}{2k+2}}_{q,\infty, \alpha} \:_{r}\Vert (M f)^{1- \frac{\alpha\beta}{2k+2}} \Vert_{\bar{q},\bar{p}} \\
 &\leq& C \Vert f \Vert^{\frac{\alpha\beta}{2k+2}}_{q, \infty, \alpha} \:_{r}\Vert M f \Vert^{1- \frac{\alpha\beta}{2k+2}}_{\bar{q}(1- \frac{\alpha\beta}{2k+2}),\bar{p}(1- \frac{\alpha\beta}{2k+2})} \\&\leq& C \Vert f \Vert^{\frac{\alpha\beta}{2k+2}}_{q, \infty,\alpha} \:_{r}\Vert M f \Vert^{1- \frac{\alpha\beta}{2k+2}}_{q,p}.
\end{eqnarray*}
Then, multiplying both sides of previous inequality by
  $\mu (B(0,r))^{\frac{1}{\alpha^{*}} - \frac{1}{\bar{q}} - \frac{1}{\bar{p}}}$, we get
\begin{eqnarray*}
& & \mu(B(0,r))^{\frac{1}{\alpha^{*}} - \frac{1}{\bar{q}} - \frac{1}{\bar{p}}} \:_{r}\Vert I_{\beta} f \Vert_{\bar{q},\bar{p}} \\
 &\leq&  C \Vert f \Vert^{\frac{\alpha\beta}{2k+2}}_{q,\infty, \alpha} (\mu (B(0,r)))^{\frac{1}{\alpha^{*}} - \frac{1}{\bar{q}} - \frac{1}{\bar{p}}} \:_{r}\Vert M f \Vert^{1- \frac{\alpha\beta}{2k+2}}_{q,p} \\
&\leq&  C \Vert f \Vert^{\frac{\alpha\beta}{2k+2}}_{q,\infty, \alpha} \mu (B(0,r))^{(\frac{1}{\alpha} - \frac{1}{q} - \frac{1}{p})(1- \frac{\alpha\beta}{2k+2} )} \:_{r}\Vert M f \Vert^{1- \frac{\alpha\beta}{2k+2}}_{q,p} \\
&\leq&  C \Vert f \Vert^{\frac{\alpha\beta}{2k+2}}_{q,\infty,\alpha} [\mu (B(0,r))^{\frac{1}{\alpha} - \frac{1}{q} - \frac{1}{p}} \:_{r}\Vert M f \Vert_{q,p} ]^{1- \frac{\alpha\beta}{2k+2}}.
\end{eqnarray*}
Hence,
$$\Vert I_{\beta}f \Vert_{\bar{q},\bar{p},\alpha^{*}} \leq C  \Vert f \Vert^{\frac{\alpha\beta}{2k+2}}_{q,\infty,\alpha}\Vert Mf \Vert^{1-\frac{\alpha\beta}{2k+2}}_{q,p,\alpha}$$
and by Proposition \ref{maxi},
\begin{eqnarray}\label{rel 6}
\Vert I_{\beta}f \Vert_{\bar{q},\bar{p},\alpha^{*}} \leq C  \Vert f \Vert^{\frac{\alpha\beta}{2k+2}}_{q, \infty,\alpha}\Vert f \Vert^{1-\frac{\alpha\beta}{2k+2}}_{q,p,\alpha}.
\end{eqnarray}
We end the proof by combining Inequality (\ref{rel 6}) with the fact that
$$\Vert f \Vert_{q,\infty,\alpha} \leq C \Vert f \Vert_{q,p,\alpha},$$
where $C$ is a constant not depending on $f$.
\end{proof}

\begin{proof}[Proof of Theorem \ref{theo 2}]
Let  $ 1 < \alpha < p < \infty$, $0<\beta< \frac{2k+2}{\alpha}$ and $f \in (L^{1}, L^{p})^{\alpha} (\mu)$.
We assume that  $f\ne 0$ $\mu$-almost everywhere, otherwise the result is obvious.
From (\ref{rel 7}) we have
\begin{eqnarray*} \label{rel 8}
\Vert I_{\beta}f \Vert_{(L^{\bar{q},\infty}, L^{\bar{p}})^{\alpha^{*}} (\mu) } \leq C \Vert f \Vert^{\frac{\alpha\beta}{2k+2}}_{1,\infty,\alpha} \Vert (M f)^{1- \frac{\alpha\beta}{2k+2}} \Vert_{(L^{\bar{q},\infty}, L^{\bar{p}})^{\alpha^{*}} (\mu) }.
\end{eqnarray*}
Now, let $r>0$. Then
\begin{eqnarray*}
& &  \Vert \Vert (M f)^{1- \frac{\alpha\beta}{2k+2}} \tau^{\frac{1}{\bar{q}}}_{(- \cdot)}\chi_{B_r} \Vert_{L^{\bar{q},\infty}(\mu)}   \Vert_{\bar{p}} \\
 &= &  \left\Vert \left\Vert \left[(Mf) \tau^{\frac{1}{\bar{q} (1- \frac{\alpha\beta}{2k+2})}}_{(- \cdot)}\chi_{B_r} \right]^{1- \frac{\alpha\beta}{2k+2}} \right\Vert_{L^{\bar{q},\infty}(\mu)}  \right \Vert_{\bar{p}} \\
  &=& \left\Vert \left\Vert (Mf) \tau^{\frac{1}{\bar{q} (1- \frac{\alpha\beta}{2k+2})}}_{(- \cdot)}\chi_{B_r}  \right\Vert^{1- \frac{\alpha\beta}{2k+2}}_{L^{\bar{q}(1- \frac{\alpha\beta}{2k+2}), \infty}(\mu)}   \right\Vert_{\bar{p}} \\
  &=& \left\Vert \left\Vert (Mf) \tau_{(- \cdot)}\chi_{B_r}  \right\Vert_{L^{1,\infty}(\mu)}   \right\Vert^{1- \frac{\alpha\beta}{2k+2}}_{\bar{p}(1- \frac{\alpha\beta}{2k+2})} \\
    &=& \left\Vert \left\Vert (Mf) \tau_{(- \cdot)}\chi_{B_r}  \right\Vert_{L^{1, \infty}(\mu)}   \right\Vert^{1- \frac{\alpha\beta}{2k+2}}_{p}.
\end{eqnarray*}
Since  $\frac{1}{\alpha^{*}} - \frac{1}{\bar{q}} - \frac{1}{\bar{p}} = (\frac{1}{\alpha} - 1 - \frac{1}{p})(1- \frac{\alpha\beta}{2k+2})$ then
$$\mu (B(0,r))^{\frac{1}{\alpha^{*}} - \frac{1}{\bar{q}} - \frac{1}{\bar{p}}} \Vert \Vert (M f)^{1- \frac{\alpha\beta}{2k+2}} \tau^{\frac{1}{\bar{q}}}_{(- \cdot)}\chi_{B_r} \Vert_{L^{\bar{q},\infty}(\mu)}   \Vert_{\bar{p}}$$ $$ =  \mu (B(0,r))^{(\frac{1}{\alpha} - 1 - \frac{1}{p})(1- \frac{\alpha\beta}{2k+2} )}\left\Vert \left\Vert (Mf) \tau_{(- \cdot)}\chi_{B_r}  \right\Vert_{L^{1,\infty}(\mu)}   \right\Vert^{1- \frac{\alpha\beta}{2k+2}}_{p} $$ $$  = \left[ \mu (B(0,r))^{\frac{1}{\alpha} - 1 - \frac{1}{p}}\left\Vert \left\Vert (Mf) \tau_{(- \cdot)}\chi_{B_r}  \right\Vert_{L^{1,\infty}(\mu)} \right\Vert_{p} \right]^{1- \frac{\alpha\beta}{2k+2}}. $$
Thus,
$$\Vert (M f)^{1- \frac{\alpha\beta}{2k+2}} \Vert_{(L^{\bar{q}, \infty}, L^{\bar{p}})^{\alpha^{*}} (\mu) } =  \Vert Mf \Vert^{1- \frac{\alpha\beta}{2k+2}}_{(L^{1, \infty}, L^{p})^{\alpha} (\mu)}$$
and
\begin{eqnarray*}
\Vert I_{\beta}f \Vert_{(L^{\bar{q},\infty}, L^{\bar{p}})^{\alpha^{*}} (\mu) } \leq C \Vert f \Vert^{\frac{\alpha\beta}{2k+2}}_{1,\infty,\alpha} \Vert Mf \Vert^{1- \frac{\alpha\beta}{2k+2}}_{(L^{1,\infty}, L^{p})^{\alpha} (\mu)}.
\end{eqnarray*}
By applying Proposition \ref{weakmaxi}, the previous inequality becomes
\begin{eqnarray} \label{rel 9}
\Vert I_{\beta}f \Vert_{(L^{\bar{q},\infty}, L^{\bar{p}})^{\alpha^{*}} (\mu_{k}) } \leq C \Vert f \Vert^{\frac{\alpha\beta}{2k+2}}_{1,\infty,\alpha} \Vert f \Vert^{1-\frac{\alpha\beta}{2k+2}}_{1,p,\alpha} .
\end{eqnarray}
We end the proof by combining Inequality (\ref{rel 9}) with the fact that
$$\Vert f \Vert_{1,\infty,\alpha} \leq C \Vert f \Vert_{1,p,\alpha},$$
where $C$ is a constant not depending on $f$.
\end{proof}
\begin{proof}[Proof of Corollary \ref{cor 1}]
Let  $1 < q \leq \alpha \leq p < \infty$, $0<\beta< \frac{2k+2}{\alpha}$ and $f \in (L^{q}, L^{p})^{\alpha} (\mu)$.\\
Let $x \in \mathbb{R}$ and  $r>0$. We have
\begin{eqnarray*}
& & \frac{1}{\mu(B(0,r))^{1-\frac{\beta}{2k+2}}} \displaystyle \int_{B(0,r)} \tau_{x} \vert f \vert (z) d\mu (z) \\ &=&  \displaystyle \int_{B(0,r)} \tau_{x} \vert f \vert (z) \frac{1}{\mu(B(0,r))^{1-\frac{\beta}{2k+2}}} d\mu (z) \\
&=& \displaystyle \int_{B(0,r)} \tau_{x} \vert f \vert (z) \frac{1}{(d_{k}r^{2k+2})^{1-\frac{\beta}{2k+2}}} d\mu (z) \\
&=& d_{k}^{\frac{\beta}{2k+2}-1}\displaystyle \int_{B(0,r)} \tau_{x} \vert f \vert (z) \frac{1}{r^{2k+2-\beta}} d\mu (z) \\
&\leq& d_{k}^{\frac{\beta}{2k+2}-1}\displaystyle \int_{B(0,r)} \tau_{x} \vert f \vert (z) \vert z \vert^{\beta -2k-2} d\mu (z)\\
&\leq& d_{k}^{\frac{\beta}{2k+2}-1}\displaystyle \int_{\mathbb{R}} \tau_{x} \vert f \vert (z) \vert z \vert^{\beta -2k-2} d\mu (z).
\end{eqnarray*}
Hence,
\begin{equation}
M_{\beta}f(x) \leq d^{\frac{\beta}{2k+2}-1}_{k} I_{\beta} \vert f \vert(x). \label{ef}
\end{equation}
We end the proof by taking the norm $\Vert \cdot \Vert_{\bar{q},\bar{p},\alpha^{*}}$ on the both sides of (\ref{ef}) and by applying Theorem \ref{theo 01}.
\end{proof}
\begin{proof}[Proof of Corollary \ref{cor 2}]
The result follows from Inequality (\ref{ef}) and Theorem \ref{theo 2}.
\end{proof}


\begin{thebibliography}{MTW1}
\bibitem{DX}F. Dai and Y. Xu, \textit{Analysis on h-Harmonics and Dunkl Transforms}, ICREA and CRM, Birkh\"auser, 2015.
\bibitem{MFE}M. F. E. de Jeu, {\it The Dunkl transform,} Invent. Math., {\bf 113}, no. 1 (1993), 147-162.

\bibitem{De} L. Deleaval, {\it On the boundedness of the Dunkl spherical maximal operator}, J. Topol. Anal, 8, no. 3 (2016), 475-495.

\bibitem{D2}C.F. Dunkl, {\it Differential-difference operators associated with reflections groups}, Trans. Amer. Math. Soc., \textbf{ 311} (1989), 167-183.
\bibitem{D1}C. F. Dunkl, {\it Hankel transforms associated to finite reflection groups}, 
Contemp. Math., {\bf 138} (1992), 123-138.
\bibitem{Fe} J. Feuto, {\it Norm inequalities in some subspaces of Morrey space}, Ann.
Math. Blaise Pascal, 21 , 2 (2014), 21-37.

\bibitem{GM2}V. S. Guliyev, Y. Y. Mammadov, {\it Function spaces and integral operators for Dunkl operator on the real line}, Khazar Journal of Mathematics, {\bf 4}, No 4 (2006), 17-42
\bibitem{GM}V. S. Guliyev and Y. Y. Mammadov, \textit{On fractional maximal function and fractional integrals associated with the
		Dunkl operator on the real line}, J. Math. Anal. Appl., 353, (1) (2009), 449-459.
\bibitem{GM3}V. S. Guliyev and Y. Y. Mammadov, \textit{$(Lp, Lq)$ boundedness of the fractional maximal operator associated with the Dunkl operator on the real line}, Integral Transforms Spec. Funct., 21, no. 7-8 (2010), 629–639.
\bibitem{M. G.}M. G. Hajibayov: \textit{Boundedness of the Dunkl convolution
	operators}, An. Univ. Vest Timiş., Ser. Mat.-Inform., 49 (2011), 49-67.
 \bibitem{MH} Y. Y. Mammadov and S. A. Hasanli \textit{On the Boundedness of Commutators of Dunkl-type Maximal Operator in the Dunkl-type Morrey Spaces}, Casp. J. Appl. Math. Ecol. Econ.,6 , No 1 (2018).

 \bibitem{M-H-S} Y. Y. Mammadov and S. A. Hasanli, \textit{On the boundedness of Dunkl-type fractional integral operator in the generalized Dunkl-type Morrey spaces}, Int. J. Appl. Math., 31, no. 2 (2018), 211-230.

 \bibitem{Hat} H. Mejjaoli, \textit{Generalized Lorentz Spaces and Applications}, J. Funct. Spaces Appl.,
 	Volume 2013, Article ID 302941, 14 pages
 	http://dx.doi.org/10.1155/2013/302941.
\bibitem{M}M. A. Mourou, \textit{ Transmutation operators associated with a Dunkl-type differential-difference operator on the real line and certain of their applications}, Integral Transforms Spec. Funct, 12, no. 1 (2001), 77-88.	
\bibitem{NF} P. Nagacy and J. Feuto, \textit{ Maximal operator in Dunkl-Fofana spaces}, Adv. Pure Appl. Math., 12, no. 2 (2021), 30-52.
\bibitem {R4} M. R\"osler, \textit{Bessel-type signed hypergroups on $\mathbb R$}, In: H. Heyer and A. Mukherjea, Eds., Probability measures on groups and related structures World Scientific, Singapore City, (1995), 292-304.
\bibitem{S}F. Soltani, \textit{ $L^{p}$-Fourier multipliers for the Dunkl operator on the real line,} J. Funct. Anal., 209, no. 1 (2004), 16-35.
\end{thebibliography}
\end{document}